\documentclass[11pt,reqno]{amsart}
\usepackage{amssymb,amsfonts,amsthm,amscd,stmaryrd,dsfont,esint,upgreek,constants,todonotes}
\usepackage[mathcal]{euscript}
\usepackage[latin1]{inputenc}   
\usepackage[mathcal]{euscript}
\usepackage{graphicx,color}
\numberwithin{equation}{section}
\newcommand{\N}{\mathbb N}
\newcommand{\Z}{\mathbb Z}

\newcommand{\R}{\mathbb R}
\def\E{\mathbb E}
\def\P{\mathbb P}
\newconstantfamily{C}{symbol=C}
\newconstantfamily{c}{symbol=c}
\newconstantfamily{O}{symbol=\Omega}

\newcommand{\T}{\mathbb{T}}

\numberwithin{equation}{section}
\newtheorem{thm}{Theorem}[section]
\newtheorem{lem}[thm]{Lemma}
\newtheorem{cor}[thm]{Corollary}

\theoremstyle{definition}

\def\smallnegint{\mathop{\int\mkern-13mu
        \raise.5ex\hbox{${\scriptscriptstyle\diagup}$}}\nolimits}
\def\ds{\displaystyle}

\def\ep{\varepsilon}

\def\F{{\mathcal F}}

\def\ssetminus{\,\raise.4ex\hbox{$\scriptstyle\setminus$}\,}
\def \lg{\langle}
\def \rg{\rangle}
\newcommand{\be}{\begin{equation}}
\newcommand{\ee}{\end{equation}}

\renewcommand{\bar}{\overline}
\renewcommand{\tilde}{\widetilde}
\renewcommand{\hat}{\widehat}

\begin{document}
\title[Perturbation problems in the theory of homogenization  of HJ equations]{Perturbation problems in homogenization of Hamilton-Jacobi equations}
\author[Pierre Cardaliaguet, Claude Le Bris  and Panagiotis E. Souganidis]
{Pierre Cardaliaguet, Claude Le Bris and Panagiotis E. Souganidis}
\address{Ceremade, Universit\'e Paris-Dauphine,
Place du Maréchal de Lattre de Tassigny, 75775 Paris cedex 16 - France}
\email{cardaliaguet@ceremade.dauphine.fr }
\address {Ecole des Ponts and Inria,
6 -8 avenue Blaise Pascal,  Cite Descartes, Champs-sur-Marne,      
77455 Marne La Vallee cedex 2 - France}
\email{lebris@cermics.enpc.fr}  
\address{Department of Mathematics, University of Chicago, Chicago, Illinois 60637, USA}
\email{souganidis@math.uchicago.edu}
\vskip-0.5in 
\thanks{\hskip-0.149in Cardaliaguet was partially supported by the ANR (Agence Nationale de la Recherche) project  ANR-12-BS01-0008-01. Souganidis was partially supported by the National Science
Foundation Grants DMS-1266383 and DMS-1600129. Part of the work was completed during Le Bris visits to the University of Chicago.}
\dedicatory{Version: \today}


\maketitle      


\begin{abstract}
\smallskip

This paper is concerned with the behavior of the ergodic constant associated with convex and superlinear Hamilton-Jacobi equation in a periodic environment which is perturbed  either by  medium with increasing period or by a random Bernoulli perturbation with small parameter.  We find a first order Taylor's expansion  for the ergodic constant which depends on the dimension $d$. When $d=1$ the first order term is non trivial, while for all $d\geq 2$ it is always $0$. Although such questions have been looked at in the context of linear 
uniformly elliptic homogenization, our results are the first of this kind in nonlinear settings. Our arguments, which   rely on  viscosity solutions  and the weak KAM theory, also raise several new and challenging questions.

\end{abstract}

\section{Introduction}

\noindent The paper is concerned with the behavior of the ergodic constant associated with convex and superlinear Hamilton-Jacobi (HJ for short) equations in a periodic environment which is perturbed  either by  medium with increasing period which is a multiple of the original one  or by a random Bernoulli perturbation with small parameter. We find  a first-order Taylor's expansion  for the ergodic constant which depends on the dimension $d$. When $d=1$ the first order term is non trivial, while for all $d\geq 2$ it is always $0$. Our results are the first 
of this kind for  nonlinear problems. The arguments, which rely on  viscosity solutions  and the weak KAM theory, also raise several new and challenging questions.
\smallskip

The motivation for this work came from the  recent studies by Anantharaman and Le Bris \cite{ALB1, ALB2} and Duerinckx and Gloria \cite{DG}, who considered similar questions for linear uniformly elliptic operators (and systems in \cite{DG}). The former paper considered Bernoulli perturbations of a periodic environment, while the  latter reference, which complemented and generalized the work of the former, considered Bernoulli perturbations of a stationary ergodic medium and provided, taking strong advantage of the linearity of the equation,  a full expansion.

\smallskip

Loosely speaking the aim of homogenization is to replace  a possibly complicated
heterogeneous medium  with a homogeneous environment that shares the same macroscopic
properties. In concrete models (equations) it allows to  eliminate  the fine scale up to an error which is
controlled by the size of fine scale as compared to the macroscopic size. 

\smallskip

From the modeling point of view, assuming that the medium is  periodic is a rather rigid and idealistic assumption and  somewhat remote from actual settings. Indeed, in view of the industrial process they are produced by, manufactured media, such as composite materials, can  be considered, under reasonable conditions, to  be periodic or at least ``approximately'' periodic. 
However, natural media, such as the subsoil, have no reason whatsoever to be periodic. Periodicity is then a mathematical idealization, or artifact, that might lead to inaccurate results.

\smallskip

A well established option is then to consider the medium to be random, and, more precisely, stationary ergodic. This  assumption conveniently makes up for the absence of periodicity, and, actually, includes periodicity as a particular case. The mathematical theory of random homogenization, both quantitative and qualitative, born in the early 1970s, has seen an enormous growth over the past fifteen years. However in spite of the appeal the theory, its application to actual media for real applications and, in particular, numerical simulations, remains a challenging issue. Random homogenization, and all approaches that derive from it, may indeed be computationally prohibitively expensive, even for the simplest possible equations arising, for instance, in the engineering sciences.  A compromise between the economical but idealistic periodic  and the more general but extremely costly random settings is to consider small random perturbations of periodic scenarios. The response of the medium in terms of this small perturbation, that is the modification of the homogenized limit in the presence of the small random perturbation,  is intuitively expected to be easier to evaluate. This was shown to be indeed true in the case of homogenization of linear elliptic equations in \cite{ALB1, ALB2}. A formal derivation of the first-order perturbation and numerical experiments performed there confirmed that it is possible, at a much reduced computational price, to approximate the homogenized limit of the random problem. As mentioned above, the approach has then been proven to be rigorous, and extended to all orders of perturbation, in a subsequent publication \cite{DG}.
\smallskip

In order to convey to the reader the flavor of the mathematical mechanism in action, we consider the following  simplistic setting, which can be thought as a computational model for the whole space $\R^d$. Let 
$F_{per}(x):=\sum_{k\in \Z^d} v(\cdot-k)$
be a $\Z^d$-periodic function that repeats itself within a presumably extremely large box of size $R$, and 
assume that a certain output,  $\mathcal S_{per}$, is computed from it. In the specific case addressed in~\cite{ALB1, ALB2}, $F_{per}$  and  $\mathcal S_{per}$ were respectively the matrix valued coefficient $A_{per}$ of the linear elliptic operator $-\hbox{\rm div} (A_{per}(./\varepsilon)\,\nabla)$ 
and the 
matrix $\overline A$ replacing $A_{per}$ in the homogenized limit. 
In this paper,  the function $F_{per}$ is the periodic Hamiltonian of a Hamilton-Jacobi equation and the outcome $\mathcal S_{per}$ is the homogenized Hamiltonian $\overline H$. Assume now that $F_{per}$ is
perturbed by the addition of  a random function of, for example,  the form $\zeta_\eta(x):=\sum_{k\in \Z^d} X_k \zeta(\cdot-k)$, where the $X_k$'s are Bernoulli random variables of a small parameter $\eta$, which are all independent from one another. 
Intuitively, at first order in $\eta$,  the perturbation experienced by  $F_{per}$ consists of adding exactly one  $\zeta$ at each possible location within the large box of size $R$. The probability of having two distinct non zero variables $X_k$  is of order $\eta^2$, a term negligible with respect to the first order  term in $\eta$. The perturbation of the outcome $\mathcal S$ with respect to the outcome $\mathcal S_{per}$ 
can therefore be calculated using only  the configurations of the  periodic medium perturbed in one random location.
Of course, the above argument is formal in many respects. For the rigorous result, we  must consider the whole infinite space $\R^d$ instead of a large box of finite size and need to prove the fact that all other configurations than those with exactly one non zero $X_k$ do not contribute to the asymptotics at first order. But the underlying idea remains.  This general discussion is made more precise below.

\smallskip

We describe next in a somewhat informal way the results of the paper. The actual statement need hypotheses which will be given in Section~2.  
\smallskip

Let $H:=H(p,x)$  be a Hamiltonian which is 
coercive in $p$ and $\Z^d-$periodic in $x$. It was shown by  Lions, Papanicolaou and Varadhan 
\cite{lions1987homogenization} that there exists
a unique $\overline H$, often referred to as the effective Hamiltonian or the ergodic constant,  such that the cell problem
\begin{equation}\label{cp}
H(D\chi,x)= \overline H \  {\rm in } \  \R^d,
\end{equation}
has a continuous, $\Z^d-$periodic (viscosity) solution $\chi$ known as a corrector. 
\smallskip

Correctors are obviously not unique. Throughout the paper, we make the normalization that $\chi(0)=0$. 
\smallskip

We recall that $\overline H$ is obtained 
as the uniform limit, as $\delta \to 0,$ of  $-\delta v^\delta$, where  $v^\delta$
is the unique periodic solution to the approximate cell problem 
\be\label{intro:approxcell}
\delta v^\delta+ H(Dv^\delta, x)=0  \  {\rm in } \  \R^d.
\ee

We consider two types of perturbations. The first is also periodic with increasingly large integer period. The second is  random (Bernoulli) with small intensity.  
\smallskip

%
%

In the first case the perturbed $R\Z^d$-periodic Hamiltonian  $H_R$,  with $R \in \N$,  is
\be\label{takisHR}
H_R(p,x):= H(p,x)-\zeta_R(x),
\ee
with  the $R$-periodic function $\zeta_R: \R^d \to \R$ defined as 
\be\label{takis8}
\zeta_R(x):= \sum_{k\in \Z^d} \zeta(x -Rk), 
\ee
where
\be\label{takiszeta}
 \zeta:\R^d\to \R \ \text{ is nonnegative, Lipschitz continuous and compactly supported.}
 \ee
In view of the form of  $\zeta_R$, we often refer to $\zeta(\cdot-Rk)$ as a ``bump" located at the point $Rk$. 
\smallskip

Let $\overline H_R$ be the ergodic constant associated with $H_R$. Then  there exists a continuous  $R\Z^d-$periodic solution $\chi_R$ of the  cell-problem
\be\label{cellpbHR}
H(D\chi_R,x)= \zeta_R(x)+ \overline H_R \  {\rm in } \ \R^d,
\ee
which is ``renormalized'' by  $\chi_R(0)=0$.
\smallskip

Since, as $R\to+\infty$, there are fewer bumps  in a given ball,  it is reasonable to expect that,  as $R\to+\infty$, $\overline H_R$ converges to $\overline H$.  Our goal is to obtain quantitative  information (rate, first term in the expansion) for this convergence.

%

\smallskip

In the second type of perturbation, the  randomly  perturbed Hamiltonian $H_\eta$ is given by
\be\label{takis60}
H_\eta(p,x):= H(p,x)-\zeta_\eta(x)
\ee
where 
\be\label{takis61}
\zeta_\eta(x):= \sum_{k\in \Z^d} \zeta(x-k)X_k,
\ee
with $\zeta$ satisfying  \eqref{takiszeta} and 
\be\label{takis100}
(X_k)_{k \in \Z^d}  \text{ a family of i.i.d. Bernoulli random variables of parameter $\eta$.} 
\ee

Contrary to the periodic setting, in random media  the effective Hamiltonian is not characterized by the cell-problem.   The reason is that to guarantee its uniqueness, it is necessary to have correctors which are strictly sub-linear at  infinity. As shown in Lions and Souganidis \cite{lions2003correctors},  in general, this is not possible.
\vskip.125in

The effective constant $\overline H_\eta$ is defined, for instance, through the discounted problem
$$
\delta v^{\eta,\delta}+ H_\eta(Dv^{\eta,\delta}, x)=0  \  {\rm in } \ \R^d
$$
which has  unique bounded solution $v^{\eta,\delta},$ as the almost sure limit (see Souganidis \cite{S})
$$
\overline H_\eta := \lim_{\delta\to 0} -\delta v^{\eta,\delta} (0). 
$$
Note that, as $\eta\to 0$, the probability that there is a bump in a fixed ball becomes smaller and smaller. So here again it is natural to expect that  $\overline H_\eta$ converges to $\overline H$  as $\eta\to 0$ and we want to understand at which rate this convergence holds. 
\smallskip

We establish two types of results. The first  is an estimate of the difference between $\overline H_R$ or $\overline H_\eta$ and $\overline H$, which  holds  even for more general  (almost periodic) perturbations. 
\smallskip

We prove that, if $H=H(p,x)$ is convex and coercive in $p$ and $\Z^d-$periodic in $x$, then there exists $C>0$ depending only on  $\zeta$ (see Corollary \ref{cor:1} and Corollary \ref{cor:1bis})  such that
\be\label{intro:barH-barHR}
0\leq \overline H-\overline H_R \leq CR^{-d} \  \text{for all} \  R\in \N,
\ee
\be\label{intro:barH-barHeta}
0\leq \overline H-\overline H_\eta \leq C\eta  \  \text{for all} \   \eta\in (0,1), 
\ee
and, in particular, 
\be\label{takis1}
\lim_{R\to \infty} \overline H_R=\overline H \ \text{ and} \  \lim_{\eta \to 0} \overline H_\eta=\overline H. 
\ee


The result is unusual in the homogenization of Hamilton-Jacobi equations because the perturbations do not vanish in the $L^\infty$-norm  and relies strongly on the fact that the bumps are nonnegative. In general the convergence does  not hold  otherwise; see Achdou and Tchou \cite{AcTc}, Lions \cite{Lcollege} and Lions and Souganidis~\cite{LS2017}, where we also refer  for more general statements about homogenization with fixed perturbations of periodic and random environments. 
\smallskip

We point out that \eqref{intro:barH-barHR}, \eqref{intro:barH-barHeta} and \eqref{takis1} are examples of more general statements which hold for general almost periodic or random perturbations; see Propositions \ref{prop:CValmostp} and \ref{prop:CVrandom}. 
\smallskip

In view of  \eqref{intro:barH-barHR}, \eqref{intro:barH-barHeta} and \eqref{takis1}, it is natural, and this is the second type of results in this paper, to identify the limits 
$$\lim_{R\to \infty}R^d(\overline H_R-\overline H) \ \text{and} \   \lim_{\eta \to 0} \eta^{-1}(\overline H_\eta-\overline H).$$ 
It  turns out that  is much more complicated  than proving \eqref{takis1}  and we only have a complete answer under some additional assumptions. 
\smallskip

In order to describe the results as well as to give a hint of the subtlety, we explain briefly and very informally the proof of  \eqref{intro:barH-barHR}.   Similar arguments justify \eqref{intro:barH-barHeta}.
\smallskip

We argue as if both $\chi$ and $\chi_R$ were smooth, which is not the case in general. We subtract   \eqref{cp} from \eqref{cellpbHR}, we linearize the difference around $D\chi$ assuming also that $H$ is smooth, and we use the convexity of $H$ to find
\be\label{takis3}
\overline H_R - \overline H\geq D_pH(D\chi, x)\cdot D(\chi_R -\chi) + \zeta_R,
\ee
where $D_pH$ denotes the gradient of $H$ with respect to $p$.
\smallskip

Let  $\tilde \sigma$ be the invariant measure associated with \eqref{cp}, which exists in view of the weak KAM theory  (see Fathi~\cite{Fa08}), that is,  $\tilde \sigma$ is a Borel probability measure in the unit cube $[-1/2,1/2)^d$
and 
$$
-{\rm div} \left(\tilde\sigma D_pH(D\chi,x)\right)=0.
$$
We extend  $\tilde \sigma$ by periodicity to $\R^d$ and we integrate both sides of \eqref{takis3} with respect to $\tilde \sigma$ over  $[-R/2,R/2)^d$. 
Using the fact that, for $R$ large enough, there is only the compactly supported bump $\zeta$ in the cube 
$[-R/2,R/2)^d$,  we find 
\be\label{takis4}
R^d(\overline H_R -\overline H) \geq - \int_{ \R^d} \zeta(x) d\tilde \sigma(x).
\ee
The last  inequality  not only justifies the right-hand side of \eqref{intro:barH-barHR}, but also  hints  that the limit of $R^d(\overline H_R-\overline H)$ should be $-\int_{\R^d}\zeta d\tilde \sigma$.  This  turns out to be  false. 
\smallskip

\smallskip
Indeed, under some assumptions on the minimizing Mather measure in the weak KAM formulation of \eqref{cp} which are stated informally below, we show in Theorem \ref{them:main}, that, when $d=1$,   
$$\lim_{R\to+\infty} R(\overline H_R-\overline H)=-(\int_{-1/2}^{1/2} \frac{1}{D_pH(\chi'(x), x)} dx)^{-1}\int_{\text{sppt}(\zeta)} \left(H^{-1}(\zeta(x) + \overline H, x) - H^{-1}(\overline H, x) \right)dx,$$
and, when  $d\geq2$, 
$$\lim_{R\to+\infty} R^d(\overline H_R-\overline H)=0.$$


For the proof we assume that the  invariant measure is unique, has a non vanishing rotational number and  its marginal $\tilde \sigma$ has a full support. The assumption on  $\tilde \sigma$  is strong and  holds only for specific classes of Hamiltonian. 
\smallskip

A schematic view of our strategy of proof goes as follows. Let $L$ be the convex dual of $H$ defined in \eqref{def:L}. The variational  interpretation of  \eqref{cellpbHR} and \eqref{cp}  implies that the respective correctors $\chi_R$ and $\chi$ satisfy the identities 
$$
\chi_R(x)=
\inf_{{\mathcal A}_x} \left[ \int_0^t (L(\dot \gamma(s), \gamma(s)) +\overline H_R+\zeta_R(\gamma(s))) \ ds + \chi_R(\gamma(t))\right],
$$
and 
$$\chi(x)=
\inf_{{\mathcal A}_x} \left[ \int_0^t (L(\dot \gamma(s), \gamma(s)) +\overline H) \ ds + \chi(\gamma(t))\right],$$
where ${\mathcal A}_x$ is the set of Lipschitz curves in $\gamma: [0,\infty) \to \R^d$ such that $\gamma(0)=x$. 
\smallskip

Let $\overline \gamma^x$ denote the optimal path in the expression for $\chi$, which exists in view of the assumptions on $H$. Then based on the equalities above, the difference of the two Hamiltonians $\overline H_R$ and $\overline H$ reads, for all $t>0$, as 
$$
\begin{array}{rl}
\ds t(\overline H_R-\overline H) \;= & \ds  -\inf_{{\mathcal A}_x} \left[ \int_0^t (L(\dot \gamma(s), \gamma(s))+\zeta_R(\gamma(s))) \ ds + \chi_R(\gamma(t))\right] +\chi_R(x)\\
& \qquad \ds + \int_0^t L(\dot{\bar \gamma}^x(s), \bar \gamma^x(s)) \ ds + \chi(\bar \gamma^x(t))- \chi(x).
\end{array}
$$
%
Identifying the limit of $R^d(\overline H_R-\overline H)$ therefore amounts to constructing a specific trajectory that almost minimizes the infimum problem in the right-hand side. Clearly,  that infimum is not achieved by  $\bar\gamma^x$, since the presence of  $\zeta_R$ has perturbed the original problem. However,
$\bar\gamma^x$ is expected to provide an accurate approximation of the infimum, at least far from the bumps. The actual proof consists in making this intuition precise and in understanding the behavior of the optimal trajectories near the bumps. 
\smallskip

The same strategy of proof applies to the  random perturbation. There it is necessary to construct an appropriate  random perturbation of the trajectory $\bar\gamma^x$. 
Most of the argument then aims at fixing all the necessary technicalities in the construction of that particular modified trajectory.

\smallskip
An intuitive way to explain the result  is 
that, when $d\geq 2$, the minimizers in \eqref{cellpbHR} eventually avoid the bump and stay close to those of \eqref{cp}, while, when $d=1$, they must pass through the bump. A similar interpretation can be used for the result in the random setting. 
\smallskip
 
The conclusion that, when $d\geq2$, $\overline H_R$ does not deviate much from $\overline H$ is in stark contrast with what is happening for uniformly elliptic divergence form operators where the first term in the expansion is nonzero. The heuristic explanation for this difference is that in the Hamilton-Jacobi setting information is propagated along curves which are lower dimensional objects when $d\geq2$, while  for  the elliptic problem  the information is obtained by averaging.
%
%
\smallskip

Next we describe some of the major ingredients in our analysis concentrating always for simplicity on 
the  periodic problem. An important fact is that, after a renormalization by additive constants, the correctors $\chi_R$ of the periodically perturbed cell-problem \eqref{cellpbHR} converge,  along subsequences as $R\to+\infty$ and locally uniformly in $\R^d$, 
 to  solutions $\chi_\infty$, which are no longer periodic,  of the equation 
\be\label{takis2}
H(D\chi_\infty,x)= \zeta(x)+\overline H \ {\rm in } \  \R^d;
\ee
the existence of such solutions was also proved by different methods in \cite{AcTc}, \cite{Lcollege} and \cite{LS2017}. 
\smallskip

The interesting  property of $\chi_\infty$ is that it keeps track of the perturbed problem, in the sense that, at least formally (see  Lemma \ref{lem:keykey} for a rigorous statement), 
\be\label{ineq:keykey}
0\geq \liminf_{R\to+\infty} R^d(\overline H_R-\overline H) \geq -\int_{\R^d} \lg D_pH(D\chi,x), D\chi_\infty-D\chi\rg d\tilde \sigma(x).
\ee
It follows from the invariance property of $\tilde \sigma$ that 
the right-hand side of \eqref{ineq:keykey} sees only the difference of $\chi_\infty-\chi$ at infinity. 
\smallskip

The analysis of $\chi_\infty$ is in itself a very intriguing problem. Using that $d\geq 2$ we prove in Corollary \ref{cor:cor} that there exists a constant $c$ such that $\chi_\infty$ is always above $\chi+c$,  coincides with $\chi+c$  outside of a ``cylinder",  and  tends to $\chi+c$ at infinity. This is enough to show that the right-hand side of \eqref{ineq:keykey} vanishes, which in turn proves that $R^d(\overline H_R-\overline H)$ tends to $0$.  The analysis when $d=1$ is based on a more direct argument.
The proof for the random perturbed problem relies on structure of $\chi_\infty$ as well. 
\smallskip

We continue with a rather brief summary of the history of the problem acknowledging that is really not possible to refer to all previous papers.  As already mentioned earlier the first homogenization result for Hamilton-Jacobi  equations in periodic environments was proved in \cite{lions1987homogenization}. Subsequent developments are due to Evans \cite{E1, E2} and, among others,  Majda and Souganidis \cite{MS}. The first result about the homogenization of Hamilton-Jacobi  equations in random media was obtained in Souganidis \cite{S} and  Rezakhanlou and Tarver \cite{RT}.
Other important contributions to the subject always in the context of the qualitative theory of homogenization for Hamilton-Jacobi equations are  Lions and Souganidis \cite{LS, LS3, lions2003correctors}, Armstrong and Souganidis \cite{AS1,AS2}, and Cardaliaguet and Souganidis \cite{CS1,CS2}.  Quantitative results, that is error estimates, were shown in Armstrong, Cardaliaguet and Souganidis \cite{ACS} and Armstrong and Cardaliaguet \cite{AC}.

\subsection*{Organization of the paper} The paper is organized as follows. In the next section we introduce the main assumptions and recall some well known facts from the weak KAM theory. In Section 3 we state and prove two general results  about the growth of the perturbations of the ergodic constant and make the connections with  \eqref{intro:barH-barHR} and \eqref{intro:barH-barHeta}. In Section 4 we introduce the assumptions and state and prove the asymptotic result for periodic perturbations, 
while in Section 5 we consider  random perturbations.

\subsection*{Notation and terminology} We work in $\R^d$ and we write $|x|$ for the Euclidean length of a vector $x\in \R^d$ and, for $x,y \in \R^d$, $\lg x,y\rg$ is the usual inner product.  The sets of integers and positive and nonnegative integers are respectively $\Z$, $\N$ and $\N_0$. If $k=(k_1,\ldots,k_d)\in \Z^d$, then $|k|_\infty:=\max_{1\leq i\leq d}|k_i|$. The cube centered at $x\in \R^d$ and of size $R>0$ is denoted by $Q_R(x):=x+[-R/2,R/2)^d$ and we set $Q_R:= Q_R(0)$ and $Q:=Q_1$ for simplicity. Given a finite subset $A \subset \Z^d$, $\sharp A$ denotes the number of elements of $A$.  Given a nonnegative measure $\mu$ and function $\zeta$, $\text{sppt}(\mu)$ and $\text{sppt}(\zeta)$ are  respectively their support. If $f:\R^d\to \R$ is bounded, $\text{osc} f:= \sup_{\R^d} f-\inf_{\R^d}$. If $f$ is integrable and $E\subset \R^d$ has a finite volume, we denote by $\fint_E f$ the average of $f$ on $E$, that is $\fint_E f= |E|^{-1}\int_E f$. For notational convenience, we write $A \lesssim B$,  if $A \leq CB \ \hbox{ for some } C>0.$  If $A \lesssim B$ and $B \lesssim A$, we write $A \approx B$. Given $f:[a,b] \to \R$,  $\left[f\right]_a^b:=f(b)-f(a),$ and, for all $k\in\N \cup\{\infty\}$, $C_c^k(\R^m)$ is the set of compactly supported $C^k$ real valued functions on $\R^m$.  Throughout the paper, $C$ is a constant that  may vary from line to line and depends on the Hamiltonian $H$ and the space dimension $d$, unless otherwise specified. All the Hamilton-Jacobi equations encountered in the text have to be understood in the sense of viscosity solutions \cite{CIL}. 

\subsection*{The random setting} We describe here the random setting that we use in the paper and introduce the necessary notation and terminology. 
\smallskip

The general setting is a probability space  $(\Omega,\F,\P)$ and  we write $\E[X]=\int_\Omega X(\omega)d\P(\omega)$ for the expectation of  a random variable $X\in L^1(\Omega;\R).$ 
We assume that the group $(\Z^d, +)$ acts on $\Omega$. We denote by $(\tau_k)_{k\in \Z^d}$ this action and assume that it is measure preserving,  that is, for all $k\in \Z^d$ and $A\in \F$, $\P[\tau_k A]=\P[A],$ and ergodic, that is, for any translation invariant $A\in \F$,  $\P[A]=0\  \text {or} \ 1.$
\smallskip

A process $F:\R^d \times \Omega \to \R$ is said to be $\Z^d$-stationary if, for all $k\in \Z^d$, $F(x+k, \omega)=F(x, \tau_k \omega)$ almost everywhere in $x$ and almost surely in $\omega.$

\smallskip
The ergodic theorem says that, if $F\in L^\infty(\R^d;L^1(\Omega))$ is stationary, then, as $N\to\infty$,
$$ \frac{1}{(2N+1)^d}\sum_{|k|_\infty \leq N}F(x,\tau_k\omega) \to \E(F(x,\cdot)) \ \text{for any} \ x\in \R^d \ \text{and \ almost surely in $\omega$.}$$
\smallskip

Finally we remark, although we will not be making use of this in the paper, that almost periodic functions can be thought as stationary functions in an appropriate probability space with continuous stationary and ergodic action (translation).


%

\section{The assumptions and some basic facts}

We introduce the assumptions  on the Hamiltonian $H$ and recall some basic facts from the weak KAM theory, for which we refer to \cite{Fa08}. We discuss the one-dimensional setting in detail as well as the existence and properties of the ``corrector'' of the perturbed problem in the whole space.
\smallskip

The motivation for this presentation is to have all assumptions and their immediate consequences in one place. 
We recommend, however, that the reader skips this section the first time and returns to it as is necessary while going through the other parts of the paper. 




\subsection*{Assumptions on the Hamiltonian}
We assume that $H\in C^2(\R^d\times \R^d; \R) $ is 
\begin{equation}\label{takis5}
\begin{cases}
\text{$\Z^d-$periodic in the second variable, that is}\\[1mm] 
 H(p,x+k)= H(p,x)  \ \text{for all $p \in \R^d,  x\in \R^d \ \text{and } \ k\in \Z^d$,}
\end{cases}
\end{equation}
and 
\begin{equation}\label{takis6}
\begin{cases}
\text{strictly convex and super-linear with respect to the first one, that is}\\[1mm] 
\hskip.25in D^2_{pp}H>0 \ \text{and}  \lim_{|p|\to +\infty} |p|^{-1}H(p,x)=+\infty \ \text{uniformly in $x$}.
\end{cases}
\end{equation}
\subsection*{Facts from the weak KAM theory}
Let $L$ be  the Lagrangian associated with $H$ which, for all $(\alpha,x)\in \R^d\times \R^d$, is given by 
\be\label{def:L}
L(\alpha, x):= \sup_{p\in \R^d} \left\{ -\lg p,\alpha\rg- H(p,x)\right\}. 
\ee
We recall from the introduction  that  $\overline H$ is the effective (ergodic) constant associated with $H$, that is $\overline H$ is the unique constant such that the cell problem \eqref{cp} has a 
$\Z^d-$periodic, continuous solution $\chi$. Note that the  coercivity of $H$ yields that $\chi$ is Lipschitz continuous.

\smallskip

The weak KAM theory provides an alternative characterization for $\overline H$, namely 
\be\label{infbarH}
-\overline H= \inf_\mu \int_{\R^d\times Q} L(\alpha,x)d\mu(\alpha,x),
\ee
where the infimum is taken over the Radon  measures $\mu$ on $\R^d \times \R^d$ which are $\Z^d-$periodic in $x$, have  weight $1$ on $\R^d\times Q$ and are closed, that is, for all  $\Z^d$-periodic  $\phi\in C^\infty(\R^d),$
$$
\int_{\R^d\times Q} \lg D \phi(x),\alpha \rg d\mu(\alpha,x)=0. 
$$
Throughout the text, we often write  $\tilde \mu$ and  $\tilde \sigma$ for an optimal measure in the minimization problem  \eqref{infbarH} and its marginal with respect to $x$ respectively. In the context of  the weak KAM theory such  $\tilde \mu$ and  $\tilde \sigma$ are called respectively a minimizing Mather measure and  a  projected minimizing measure. Note that the restriction of $\tilde \sigma$ to $Q$ is a probability measure. 
\vskip.075in

We use the following well known facts from the weak KAM theory (see \cite{Fa08}):  
\[ \text{any corrector $\chi$  is $\tilde \sigma-$a.e. differentiable,}\]
\[\text{$\tilde \mu$ is  the image of the measure $\tilde \sigma$ by the map $x\to \left( D_pH(D\chi(x),x),x\right)$,}\]
and
\[\text{$\tilde \sigma$ is an invariant measure for the flow generated by the vector field $x\mapsto -D_pH(D\chi(x),x)$,}\]
that is
\be\label{takisinvariant}
{\rm div}\left( \tilde \sigma D_pH(D\chi,x)\right)=0 \ \text{in the sense of distributions in } \ \R^d.
\ee
\subsection*{An assumption on the Mather measure and its consequences}  In order to prove the asymptotic results and, in particular, the existence of the limits discussed in the Introduction, we need to further assume  that 
\be\label{takisadditional}
\begin{cases}
\text {the projected Mather measure $\tilde \sigma$ associated with $H$ is unique,}\\[1mm]
\text{has a nonzero rotation number $e$  and  full support in $\R^d$.}
\end{cases}
\ee 
The assumption of the full support in $\R^d$ is written as 
$$
{\rm sppt}(\tilde \sigma)=\R^d,
$$
and the nonzero rotation number $e$ is given by 
\be\label{hyp:e}
e:=\int_Q -D_pH(D\chi(x),x)d\tilde \sigma(x)\neq 0.
\ee
The first two conditions in \eqref{takisadditional}, that is the uniqueness of $\tilde \sigma$ and existence of a nonzero rotation number, are rather mild. For example, if $H(p,x)= \tilde H(p+\bar p, x)$ for some Hamiltonian $\tilde H$ satisfying \eqref{takis5} and \eqref{takis6} and some $\bar p\in \R^d$, then  $\tilde \sigma$ is unique  for a ``generic'' $\bar p$ and the nonzero   rotation number exists for $\bar p$ large enough; see \cite{Fa08}. That the projected Mather measure has full support in $\R^d$  is a much stronger assumption and only holds under restrictive structure conditions. 
\smallskip

We continue listing several consequences of \eqref{takisadditional} that are used in the rest of the paper. We refer to \cite{Fa08} and references therein for the proofs. 
\smallskip

Since the projected Mather measure has a full support, 
\be\label{taxisaubry}
\text{  the projected Aubry set is $\R^d$.}
\ee 
It then follows that 
\be\label{takisregularity}
\text{ any corrector $\chi$ is of class $C^{1}$, and, thus,  also $C^{1,1}.$}
\ee
The strict convexity of the Hamiltonian also implies that 
\be\label{takisuniqueness}
\text{the correctors are unique up to an additive constant.}  
\ee
Indeed, if $\chi$ and $\tilde \chi$ are two correctors, subtracting their respective equations and using the strict convexity we find, for some $C>0$, 
$$
0= H(D\tilde \chi,x)-H(D\chi,x) \geq \lg D_pH(D\chi,x), D(\tilde \chi-\chi)\rg + C |D(\tilde \chi-\chi)|^2. 
$$
Multiplying  the inequality above by $\tilde \sigma$, integrating over $Q$  with respect to $\tilde \sigma$ and integrating  by parts using the periodicity and the fact that   $\tilde \sigma$ is an invariant measure, that is \eqref{takisinvariant} holds,  we obtain
$$
0\geq \int_Q |D(\tilde \chi-\chi)|^2d\tilde \sigma(x). 
$$
Thus the continuous maps $D\tilde \chi$ and $D\chi$ agree on a dense subset of $\R^d$ and therefore everywhere. 
\smallskip

In view of \eqref{takisregularity}, we  can define the flow $\bar \gamma^x:\R\to \R^d$ of optimal trajectories for any initial position $x\in \R^d$ by  
$$
\dot{\bar \gamma}^x(t) = - D_pH(D\chi(\bar \gamma^x(t)), \bar \gamma^x(t) ) \ \text{ for } \ t\in \R  \ \text{ and}  \ \bar \gamma^x(0)=x;
$$
we note that the map $x\mapsto \bar \gamma^x(t)$ is continuous for any $t$. 
\smallskip

We recall now that the optimality of $\bar \gamma^x$ implies that, for all $t\geq 0$,  
\be\label{takisoptimal}
\chi(x)=\int_0^t (L(\dot{\bar \gamma^x}(s), \bar \gamma^x(s))+ \overline H) ds + \chi(\bar \gamma^x(t)).
\ee

The uniqueness of the projected Mather measure implies that it is actually ergodic. As a result, 
for $\tilde \sigma-$ a.e. $x\in \R^d$,   we have 
\be\label{RotNumber}
\lim_{t\to \pm \infty} \frac{  \bar \gamma^x(t)}{t} = \int_{\T^d} -D_pH(D\chi(x),x) d\tilde \sigma(x)=e.
\ee
As a matter of fact we will see below (Lemma \ref{lem:avoid}), that, as  a consequence of the uniqueness of $\tilde \sigma$,  \eqref{RotNumber}  actually holds  for all $x\in \R^d$.
\smallskip

We present now a  simple example satisfying \eqref{takis5}, \eqref{takis6} and \eqref{takisadditional}. Let $H(p,x)= |p+\bar p|^2$ for some non rational direction $\bar p\in \R^d$. In this  case, we have $\chi=0$ and $\bar \gamma^x(t)= x+ t\bar p$. The unique invariant measure is $\tilde \sigma=1$ and $e=-2\bar p$. 
\smallskip

The KAM theory then  implies that \eqref{takisadditional} holds true for $H(p,x):= |p+\bar p|^2-V(x)$ with  $\bar p$ a Diophantine vector and $V:\R^d\to\R$ periodic, smooth and small enough. Obviously, \eqref{takis5} and \eqref{takis6} are satisfied.  
\subsection*{The one-dimensional setting when \eqref{takis5}, \eqref{takis6} and \eqref{takisadditional} hold} 
We know from \eqref{takisadditional} that the cell problem \eqref{cp} has a $\Z$-periodic solution $\chi \in C^{1,1}(\R)$  and $\int _Q \chi'(x) dx=0.$
\smallskip

The strict convexity of $H$ implies that  the inverse $H^{-1}(\cdot, x)$ of $H(\cdot,x)$ has two branches $H_\pm^{-1}(\cdot, x)$ as long as one is away from its minimum, which is the case  in view  of \eqref{takisadditional}. Since the corrector is smooth, $\chi'(x)$ must be, for all $x$, in the same branch of $H^{-1}(\cdot, x)$  and we can rewrite \eqref{cp} as an the ode
\be\label{takisode}
\chi'(x)=H^{-1}(\overline H, x),
\ee
using only one of them.  The choice of the branch, which from now we denote as $H^{-1}(\cdot, x)$,  is dictated by $\int _Q \chi'(x) dx=0.$  

%
%
\smallskip

In view of the above discussion, in any $Q_R$ with $R\in \Z$, we have 
\begin{equation}\label{takis41}
\int_{Q_R}H^{-1}(\overline H, x) dx=0.
\ee
It also follows from \eqref{takisadditional} and \eqref{takisinvariant} that the invariant measure $\tilde \sigma$ associated with the cell problem at hand has  $\Z-$periodic extension in $\R$  with density 
\be\label{takis42}
\tilde \sigma (x)= (\int_{Q} \frac{1}{D_pH(\chi'(y), y)} dy)^{-1}  \frac{1}{D_pH(\chi'(x), x)};
\ee
note that for notational simplicity we often identify the invariant measure with its density, 

\smallskip
Let $D_r H^{-1}(\cdot, x)$ denote the derivative of $r \mapsto H^{-1}(r, x) $ with respect to the first argument. It follows from 
\eqref{takisode} that 
\be\label{takis42ter}
D_r H^{-1}(\overline H, x) = \frac{1}{D_pH(\chi'(x), x)},
\ee
and, in view of  \eqref{takis42}, 
\be\label{takis42bis}
D_r H^{-1}(\overline H, x)=\left(\int_{Q} \frac{1}{D_pH(\chi'(y), y)} dy\right) \  \tilde \sigma(x).
\ee
\smallskip

We conclude with the following classical example always for $d=1$. The Hamiltonian is $H(p,x)= |p + \bar p|^2 - V(x)$ for some fixed $\bar p \in \R$ and a  $\Z$-periodic potential $V$ with $\min_{x\in Q}V(x)=0$. It is well known
that, if $|\bar p| \geq \int_Q \sqrt {V(x)} dx$, then the cell problem 
$$|\chi_x + \bar p|^2 =V(x) +\overline H \ \text{in} \ \R,$$
has a smooth $\Z-$periodic solution for $\overline H$ given by $|\bar p| = \int_Q \sqrt{V(x) +\overline H}$. This last expression and the sign of $\bar p$ identify the branch of the $\sqrt {\cdot}$ that we need to choose.
\subsection*{A corrector $\chi_\infty$ of the perturbed problem in the whole space}  
An important ingredient in our analysis is the construction of a ``perturbed corrector'' $\chi_\infty$, that is a solution to \eqref{takis2}, which, as it turns out (see  Lemma \ref{lem:keykey}), 
keeps track of the difference between $\overline H_R$ and $\overline H$ as $R\to \infty$.

%
\smallskip

The first step in finding  $\chi_\infty$ is to obtain  independent of $R$ sup- and Lipschitz bounds for the $R\Z^d-$ periodic solutions $\chi_R$ to \eqref{cellpbHR}; recall that we always consider $R\in \N$.
%
\begin{lem} Assume \eqref{takis5}, \eqref{takis6} and \eqref{takiszeta}. There exist solutions  $\chi_R$ of the perturbed cell problem \eqref{cellpbHR}
such that
$$
\|\chi_R\|_\infty+\|D\chi_R\|_\infty \lesssim 1.
$$
\end{lem}
\begin{proof} The gradient bound follows immediately  from the coercivity of $H$ and  holds for any solution of the cell problem. For the $L^\infty$-bound, we consider the approximate cell problems \eqref{intro:approxcell} and 
\be\label{ApproxCellPbR}
\delta v^\delta_R +H(Dv^\delta_R, x)=\zeta_R \ {\rm in } \ \R^d,
\ee
which are  respectively  $\Z^d$ and $R\Z^d$ periodic. Since, for any $\ep>0$, $v^\delta_R-\ep$ is a  strict subsolution to  \eqref{intro:approxcell} in $\R^d\backslash {\rm sppt}(\zeta_R)$, the maximum of $v^\delta_R-\ep-v^\delta$, if positive, can only be reached at some $x_\ep \in {\rm sppt}(\zeta_R)$ and, in view of the periodicity, we may assume that $x_\ep\in {\rm sppt}(\zeta)$. Similarly, the minimum of $v^\delta_R-\ep-v^\delta$, if negative, is reached at a point $y_\ep \in {\rm sppt}(\zeta)$. 
\smallskip

Thus 
\be\label{takis50}
{\rm osc}_{\R^d}(v^\delta_R-v^\delta) \leq {\rm osc}_{{\rm sppt}(\zeta)}(v^\delta_R-v^\delta) +2\ep \leq \|D(v^\delta_R-v^\delta)\|_\infty {\rm diam}({\rm sppt}(\zeta))+\ep.
\ee
Recall that the $v^\delta$'s are $\Z^d-$periodic. Moreover, in view of the assumed coercivity and bounds on $H$, the  $v^\delta$'s are Lipschitz continuous uniformly  in $\delta$.  Hence their oscillations are  bounded uniformly in $\delta$. 
\smallskip

It follows from \eqref{takis50} that  the oscillation of $v^\delta_R$ is also bounded, uniformly with respect to $R$ and $\delta$. Thus, we can extract a subsequence $\delta_n\to0$ such that $v_R^{\delta_n}-v_R^{\delta_n}(0)$ converge uniformly in $\R^d$ to a solution $\chi_R$ of the perturbed cell problem \eqref{cellpbHR} satisfying the uniform  $L^\infty$ and Lipschitz bounds. 
\smallskip

Up to a subsequence, we can assume that,  as $R\to \infty$,  the $\chi_R$'s converge locally uniformly to some $\chi_\infty$, which is no longer periodic,  solving \eqref{takis2}. 
\end{proof}
\smallskip

We discuss next  some  properties of the map $\chi_\infty$ and the optimal trajectories for $\chi$ and $\chi_\infty$ which will  be useful for the asymptotic limit of the random perturbation in Section~5. The proof of  Corollary~\ref{cor:cor} is presented at the end of Section~4, since it is there that all the necessary machinery is been developed.
\smallskip

\begin{lem}\label{lem:avoid} In addition to \eqref{takis5} and \eqref{takis6},  assume  that the minimizing Mather measure is unique and $e\neq 0$ in \eqref{RotNumber}.  For any $C>0$ and any $\ep>0$, there is a time $T_0=T_0(C,\ep)>0$ such that, if $\gamma$ is such that
\be\label{condgamma}
\int_0^T \left( L(\dot \gamma(t), \gamma(t))+\overline H\right)dt \leq C \ \text{ for all} \ T\geq T_0, 
\ee
then 
$$
\left| \frac{\gamma(t)-\gamma(0)}{t}-e\right|\leq \ep  \ \text{ for all} \ t\geq T_0.
$$ 
\end{lem}
We remark that the coercivity assumption and \eqref{condgamma} imply that $ \|\dot \gamma\|_{L^2}$ and, hence, $|\gamma(t)-\gamma(s)|$ are uniformly bounded on bounded (time) intervals of size less than $T_0$. The lemma above, provides a bound on $|\gamma(t)-\gamma(s)|$  for time intervals of length larger than $T_0$. 
\smallskip

An immediate consequence of Lemma~\ref{lem:avoid}, which is used in the analysis of the asymptotic limits,  is stated in the next corollary. Its proof, which is essentially a restatement of the conclusion of  Lemma~\ref{lem:avoid},  is omitted. 
\begin{cor}\label{cor:takis} In addition to \eqref{takis5} and \eqref{takis6},  assume that the minimizing Mather measure is unique and $e\neq 0$ in \eqref{RotNumber}. For any $C, \theta>0$, there exist $T_0=T_0(C,\theta)>0$ and $R_0=R_0(C)>0$, such that, if $\gamma$  satisfies  the bound \eqref{condgamma} with the given $C$, then 
$$
\inf_{s\geq t} \lg \gamma(s)-\gamma(t), e \rg \geq -R_0 \ \  {\rm and} \ \  \inf_{s\geq t+T_0}\lg \gamma(s)-\gamma(t), e \rg \geq \theta \  \text{for all}  \ t\geq 0. 
$$
In particular, there exists  $K=K(C)>0$  such that any $\gamma$ satisfying $\lg e,\gamma(0)\rg \geq K$ and \eqref{condgamma} avoids the support of $\zeta$ for any positive time. 
\end{cor}

\begin{proof} [The proof of Lemma~\ref{lem:avoid}] Let $\gamma_n$ be a sequence of trajectories satisfying \eqref{condgamma}. In view of the  periodicity, we may  assume without loss of generality that $\gamma_n(0)\in Q$. 
\smallskip

Let $\mu_{n,T}$ be the occupational measure on $\R^d\times\R^d$ which is periodic in space and defined, for all $\phi=\phi(\xi,x) \in C^\infty(\R^d\times\R^d)$ which are periodic with respect to $x$, by 
$$
\int_{\R^d\times Q} \phi(\xi,x) d\mu_{n,T}(\xi,x):= \frac{1}{T}\int_0^T \phi(\dot \gamma_n(t), \gamma_n(t))dt.
$$
It follows from  \eqref{condgamma} that 
$$
\int_{\R^d\times Q}  L(\xi,x) d\mu_{n,T}(\xi,x) \leq -\overline H+\frac{C}{T},
$$
and, hence, in view of the coercivity of $L$,  the family $(\mu_{n,T})_{n \in \N, T\geq 0}$ is tight. 
\smallskip

Then, as $n \ \text{and} \ T \to \infty$, there exists a subsequence  of $\mu_{n,T}$ (for simplicity we do not change the notation of the subsequence)  that  converges weakly  to a measure $\mu$ satisfying 
$$
\int_{\R^d\times Q}  L(\xi,x) d\mu(\xi,x) \leq -\overline H.
$$
Note also that, for any $\Z^d$-periodic $\phi \in C^1(\R^d)$, 
$$
\begin{array}{rl}
\ds \int_{\R^d\times Q} \lg D\phi(x), \xi\rg d\mu(\xi,x)= & \ds \lim_{n,T\to \infty} \int_{\R^d\times Q} \lg D\phi(x), \xi\rg \mu_{n,T}(\xi,x) \\[4mm]
=& \ds \lim_{n,T\to \infty}  \frac{\phi(\gamma_n(T))-\phi(\gamma_n(0))}{T}=0,
\end{array}
$$
that is $\mu$ is also closed, and, hence,   a Mather minimizing measure and, in view of the assumed uniqueness, the entire family $(\mu_{n,T})_{n \in \N, T\geq 0}$ converges to $\mu$ as  $n, T \to \infty$. 
Moreover, 
$\mu$ is the image of $\tilde \sigma$ by the map $x\mapsto (-D_pH(D\chi(x),x),x)$. 
\smallskip

In particular,  as $n \ \text{and} \ T \to \infty$, 
$$
\frac{\gamma_n(T)-\gamma_n(0)}{T}= \frac{1}{T} \int_0^T \dot \gamma_n(t)dt= \int_{\R^d\times Q} \xi d\mu_{n,T}(\xi,x)\to \int_{Q} -D_pH(D\chi(x),x) d\tilde \sigma(x)=e. 
$$
The claim  then follows from the assumption that $e\neq 0$.
\end{proof}

Finally the proof of Theorem~\ref{them:main} yields the  following  partial description of $\chi_\infty$ which is of independent interest. 
\begin{cor}\label{cor:cor}   Let ${\mathcal O}$ be the open set of points such that $(\bar \gamma^x)_{t\in \R}$ does not intersects ${\rm sppt}(\zeta)$.  There exists  $c\in \R$ such that 
\begin{itemize}
\item[(i)] $\chi_\infty = \chi+c$ in ${\mathcal O}$, 
\item[(ii)] $\chi_\infty\geq \chi+c$ in $\R^d$, 
\item[(iii)]  there exists  $K\geq 0$ such that,  if $\lg x,e\rg \geq K$,  then $\chi_\infty (x)= \chi(x)+c$,
\item[(iv)] for any $\ep>0$  there exists $K_\ep>0$ such that, if $\lg x,e\rg \leq -K_\ep$, then 
$(\chi_\infty-\chi)(x) \leq c+\ep$.
\end{itemize}
\end{cor}
The proof is presented at the end of Section~4.

\section{ The growth of the perturbed ergodic constant}
Given a Hamiltonian $H$ we consider perturbations of the form $H(p,x)-f(x),$ where $f$ is a non negative potential and prove that, under assumptions on the potential,  the difference of the corresponding effective constants can be controlled by some ``average" of $f$. 
\smallskip

We present two results, one for almost periodic  and one for random media. 
Then  we describe the relationship with the two examples in the introduction and prove \eqref{intro:barH-barHR} and \eqref{intro:barH-barHeta}.

\subsection*{Almost periodic perturbations}
Let  
\be\label{takis7}
f:\R^d\to \R \ \text{be nonnegative, bounded uniformly continuous and almost periodic},
\ee 
and recall that almost periodicity implies  the existence of the average 
$$
\fint_{\R^d} f=\lim_{R\to+\infty} \fint_{Q_R} f.
$$
Given $H$ satisfying \eqref{takis5} and \eqref{takis6}, we consider the perturbed Hamiltonian
$$H_f(p,x):=H(p,x)-f(x),$$
and note that, for any $p\in \R^d$, $x\to H_f(p,x)$ is almost periodic.

\smallskip

We recall (see Ishii \cite{ishii}) that  the ergodic constant $\overline H_f$ associated with $H_f$  is obtained as the uniform in $\R^d$ limit, as $\delta\to 0$, of $-\delta v^\delta$,
 where $v^{\delta}$ is the the unique bounded viscosity solution to 
$$
\delta v^{\delta} +H_f(Dv^{\delta},x)=0 \  {\rm in } \ \R^d.
$$
It is  straightforward implication of the comparison principle of viscosity solutions that 
$$
0\leq \overline H-\overline H_f \leq \|f\|_\infty.
$$
This estimate does not depend on the almost periodicity of $f$ and, hence, is not useful here since it does not ``see'' the averaging that is taking place. 
\smallskip

To obtain a more precise  estimate of the difference $\overline H-\overline H_f,$
we introduce the auxiliary quantity 
$$
\hat f(x):= \limsup_{N\to+\infty} N^{-d} \sum_{Q_1(k)\subset Q_N}f(x+k),
$$
and note that, since $f$ is uniformly continuous,  $\hat f$ is a $\Z^d-$periodic and continuous map, and, 
moreover, and this is very important, $\|\hat f\|_\infty$ is in general much smaller than $\|f\|_\infty$. 

\smallskip
%

To illustrate the difference between $\|\hat f\|_\infty$ and $\|f\|_\infty$, we discuss  the example we considered in the introduction, that is the periodic perturbation $\zeta_R$ given by \eqref{takis8}, which is obviously almost periodic,
and we estimate $\|\hat \zeta_R\|_\infty.$
\smallskip

 

%
 
Let $K>0$ such that the support of $\zeta$ is contained in $Q_K$. Then,  for any $x\in Q$ and $N\in \N$, we have
$$
N^{-d} \sum_{Q_1(k)\subset Q_N} \zeta_R(x+k) \leq N^{-d}  \sum_{Q_1(k)\subset Q_N} \sum_{k'\in \Z^d} \|\zeta\|_\infty {\bf 1}_{\{|k+Rk'|\leq K+1\}}.
$$
\smallskip

It follows that,  for $N$ sufficiently larger than $R$, 
$$
N^{-d} \sum_{Q_1(k)\subset Q_N} \zeta_R(x+k) \leq N^{-d} \|\zeta\|_\infty \sum_{k'\in \Z^d} \sharp\left\{k\in \Z^d: Q_1(k)\subset Q_N\cap Q_{K+1}(-Rk') \right\},
$$
and, hence, 
$$
\|\hat \zeta_R\|_\infty\lesssim R^{-d}.
$$
Note that, as $R\to+\infty$,  $\hat \zeta_R \to 0$, whereas $\|\zeta_R\|_\infty=\|\zeta\|_\infty$ is constant. 
\smallskip

The general result about the size of the perturbation is the following Theorem. 
\begin{thm}\label{prop:CValmostp} Assume \eqref{takis5}, \eqref{takis6} and \eqref{takis7}. Then 
$$
0\leq \overline H-\overline H_f \leq \|\hat f\|_\infty.
$$
\end{thm}

\smallskip

In view of the computations above, we have the following corollary. 

\begin{cor}\label{cor:1} Assume \eqref{takis5}and  \eqref{takis6} and  consider, for $R \in \N$,  the perturbation $\zeta_R$ given by \eqref{takis8}. 
Then
$$
0\leq \overline H-\overline H_R \lesssim R^{-d}.$$
\end{cor}

Theorem~\ref{prop:CValmostp} states that $\overline H_f$ is close to $\overline H$ if the almost periodic perturbation $f$ is nonnegative and small in ``average". In the two examples discussed below we show  that both conditions are sharp.
\smallskip

The assumption that $\zeta$ is nonnegative cannot  be removed.  we show this in the framework of Corollary~\ref{cor:1}. Indeed, if $H(p)=|p|^2$ and $H_{R}(p,x)=|p|^2-\zeta_R(x)$, then it is known (see \cite{lions1987homogenization}) that $\overline H=0$ and, independently of the sign of $\zeta_R$,  $\overline H_R= -\inf \zeta_R.$ 
In particular,  $\overline H_R$ does not converge to $\overline H=0$ as $R\to+\infty$,  if $\zeta$ takes negative values, since, in this case $\inf \zeta_R = \inf\zeta<0$. 
\smallskip

Next we discuss the manner in which  the perturbation is averaged. When $f$ is nonnegative, it seems reasonable to expect that $ \|\hat f\|_\infty$ can  be replaced by the  average of $f$ in Proposition \ref{prop:CValmostp}.
This is, however, also not true.  For example,  fix $\ep>0$ small and $R>0$ large and consider
the perturbation 
$$
\zeta_R(x)= \ep \sum_{k\in \Z^d} \zeta(R(x-k)),
$$
with $\zeta$  satisfying \eqref{takiszeta} and  $\zeta(0)=1$; note that this $\Z^d-$periodic perturbation differs from the one in \eqref{takis8} and thus Corollary \ref{cor:1} does not apply here. Then, uniformly on $\ep\in(0,1)$,  
$$\lim_{R\to \infty} \ds \fint_{\R^d} \zeta_R=0.$$
Moreover, if $H(p,x)=|p|^2-g(x)$,  where $g$ is continuous, then  $\overline H=-\inf g$. If, in addition, $g$ has  a unique and strict  global minimum on $Q$ at $0$ and $\ep$ is sufficiently small independently of $R$, then $\inf (g+\zeta_R) \approx \inf g+ \zeta_R(0)= \inf g+\ep$.  Hence  $\overline H_{\zeta_R}\approx -(\inf g+\ep)$ does not tend, as $R\to+\infty$, to $\overline H= -\inf g$. 

\begin{proof}[Proof of Theorem~\ref{prop:CValmostp}.] 
Since $f$ is nonnegative, the comparison argument yields  $\overline H_f\leq \overline H$.  
\vskip.075in

For the upper bound, it is convenient to regularize the problem and to consider, for $\ep>0$, the almost periodic solution $v^{\delta,\ep}$ of the  approximate  viscous  cell-problem
\be\label{takis9.9}
\delta v^{\delta,\ep}-\ep \Delta v^{\delta,\ep}+H(Dv^{\delta,\ep},x)=f \  {\rm in} \ \R^d.
\ee
We recall  that, as $\ep \to 0$,  $\delta v^{\delta,\ep} \to \delta v^\delta$ uniformly in $x$ and $\delta$, where $v^\delta$ is the solution to 
$$
\delta v^{\delta}+H(Dv^{\delta},x)=f \  {\rm in} \ \R^d.
$$
\smallskip

We also consider the solution $\chi^\ep$ of the viscous periodic cell-problem associated to $H$, that is 
\be\label{takis9.99}
-\ep \Delta \chi^{\ep}+H(D\chi^{\ep},x)=\overline H^\ep \ \text{in} \ \R^d,
\ee
as well as the associated ergodic measure which has  a continuous, strictly positive, $\Z^d-$periodic  density $\tilde\sigma^\ep$ of mass $1$ over $Q$ satisfying 
\be\label{takis10}
-\ep\Delta \tilde \sigma^\ep-{\rm div}\left( \tilde \sigma^\ep D_pH(D\chi^\ep,x)\right)=0  \ \text{in} \ \R^d.
\ee 
Finally, we recall that $\lim_{\ep \to 0} \overline H^\ep= \overline H$, while the measure $\tilde \sigma^\ep$ converges, up to subsequences,  to some Mather minimizing measure $\tilde \sigma$. 
\smallskip

Subtracting \eqref{takis9.99} from \eqref{takis9.9} and using the  convexity of $H$ we find that  $v^{\delta,\ep}-\chi^\ep$ solves  
$$
-\ep \Delta(v^{\delta,\ep}-\chi^\ep)+\delta v^{\delta,\ep} +\lg D_pH(D\chi^{\ep},x), D(v^{\delta,\ep}-\chi^\ep)\rg\leq f -\overline H^\ep  \ \text{in} \ \R^d.
$$
Multiplying the  above inequality  by $\tilde \sigma^\ep$, integrating over $Q_N$ for a large $N\in \N$ and using that $\sigma^\ep$ is an invariant measure, that is \eqref{takis10}, we find 
\be\label{ineq:jhqsbvoz}
\begin{array}{r}
\ds \delta \int_{Q_N}v^{\delta,\ep}\tilde \sigma^\ep + \int_{\partial Q_N} \lg-\ep D(v^{\delta,\ep}-\chi^\ep)+(v^{\delta,\ep}-\chi^\ep)D_pH(D\chi^{\ep},x)), \nu\rg \tilde \sigma^\ep \\
\ds \leq \int_{Q_N} f(x)\tilde \sigma^\ep(x) dx - \overline H^\ep N^d,
\end{array}
\ee
where $\nu$ is the outward unit normal at $Q_N$. 
\smallskip

Note that the  periodicity of $\tilde\sigma_\epsilon$ yields 
$$
\int_{Q_N} f(x)\tilde \sigma^\ep(x) dx=\int_{Q}  \sum_{Q_1(k)\subset Q_N} f(x+k)\tilde \sigma^\ep(x)dx,
$$
while the integrand of the integral over $\partial Q_N$ in \eqref{ineq:jhqsbvoz} is bounded. 
\smallskip

Dividing \eqref{ineq:jhqsbvoz}  by $N^d,$ letting $N\to+\infty$ and  using  Fatou's Lemma, we get
$$
\delta \fint_{\R^d}v^{\delta,\ep}\tilde \sigma^\ep \leq \int_{Q} \hat f(x)\tilde \sigma^\ep(x) dx - \overline H^\ep .
$$
Finally letting first  $\ep\to0$ and then $\delta\to 0$ yields the claim, in view of the  uniform convergence in $\R^d$ of $v^{\delta,\ep}$ to $v^\delta$ (as $\ep\to0$) and of $\delta v^\delta$ to $-\overline H_f$ (as $\delta\to0$) and by the convergence of $\tilde \sigma^\ep$ to $\tilde \sigma$ in measure (as $\ep\to0$).
\end{proof}

\subsection*{Random perturbations}
We  consider here a perturbation of the Hamiltonian $H$ by a  random potential $f$ in a probability space described in the introduction. 
%
\smallskip

We assume that 
\be\label{takis11}
\begin{cases}
f:\R^d\times \Omega\to \R \ \text{ is nonnegative,  continuous with respect to the first variable} \\[1.2mm]
\text{uniformly with respect to the second variable, and  $\Z^d-$stationary,}
\end{cases}
\ee
that is $f(x+k,\omega)= f(x, \tau_k \omega)$ for all $x\in \R^d, k\in \Z^d$ and $\omega \in \Omega$. It follows that the map $x\to \E[f(x,\cdot)]$ is a continuous and $\Z^d-$periodic function.
\smallskip


Let $\overline H_f$ be the ergodic constant associated with the Hamiltonian $H(p,x) -f(x)$ which exists (see \cite{S}) and is identified by the a.s. limit $\overline H_f := \lim_{\delta\to 0} -\delta v^{\delta} (0)$,
$v^{\delta}$ being the bounded and Lipschitz continuous, with a constant independent of $\delta$, $\Z^d$-stationary solution to the discounted problem
$$\delta v^{\delta} + H(Dv^\delta, x) -f=0 \ \text{in} \ \R^d.$$  

\begin{thm}\label{prop:CVrandom} Assume \eqref{takis5},  \eqref{takis6} and  \eqref{takis11}. Then
$$
0\leq \overline H-\overline H_f \leq \sup_{x\in \R^d}  \E[f(x,\cdot)].
$$ 
\end{thm}
%
We describe now  the particular case of the above result which was discussed in the introduction and will be further investigated  in Section \ref{sec:random}. 
\smallskip

Fix $\zeta$ satisfying \eqref{takiszeta} 
and,  for $\eta\in (0,1)$, let $(X^\eta_k)_{k\in \Z^d}$ be a family of i.i.d. Bernoulli random variables of parameter $\eta$, that is 
$$
\P[ X^\eta_0=1]=1-\P[X^\eta_0=0]= \eta.
$$
Set
$$
\zeta_\eta(x):= \sum_{k\in \Z^d} X^\eta_k \zeta(x-k)\qquad {\rm and }\qquad H_\eta(p,x):=H(p,x)-\zeta_\eta(x),
$$
and  denote by $\overline H_\eta$ (instead of $\overline H_{\zeta_\eta}$) the effective constant. 
\smallskip

In this context, Theorm~\ref{prop:CVrandom}  yields immediately the following result. 
\begin{cor}\label{cor:1bis} Assume \eqref{takis5},  \eqref{takis6}  and let $\zeta_\eta$ be defined as above.  Then,  for all $\eta\in (0,1)$, 
$$
0\leq \overline H-\overline H_\eta\lesssim\eta.
$$
\end{cor}
We continue with the:
\begin{proof}[Proof of Theorem~\ref{prop:CVrandom}.] 
Using the notation and strategy of the proof of Theorem~\ref{prop:CValmostp}, we have  
$$
\begin{array}{r}
\ds \delta \int_{Q_N}v^{\delta,\ep}\tilde \sigma^\ep + \int_{\partial Q_N} \lg-D(v^{\delta,\ep}-\chi_\ep)+(v^{\delta,\ep}-\chi^\ep)D_pH(D\chi^{\ep},x)), \nu\rg \tilde \sigma^\ep \\
\ds \leq \int_{Q_N} f(x,\omega)\tilde \sigma^\ep(x) dx - \overline H^\ep N^d.
\end{array}
$$
Note that the maps $x\mapsto \E[v^{\delta,\ep}(x)]$ and $x\mapsto \E[ f(x)]$ are  $\Z^d-$periodic. Hence taking expectation, dividing by $N^d$ and letting $N\to+\infty$ in the above inequality, we find
$$
\int_Q \delta \E[v^{\delta,\ep}] \tilde \sigma^\ep 
\leq \int_Q\E[ f]\tilde \sigma^\ep -\overline H^\ep,
$$
and, letting  first $\ep\to0$ and then $\delta \to0$, we obtain 
$$
\overline H-\overline  H_\eta  \leq  \sup_x \E[ f(x,\cdot)].
$$
The  inequality $\overline H\geq \overline H_\eta$ is an immediate consequence of the comparison principle. 
\end{proof}

\section{Sharper convergence  for the periodic perturbation}\label{sec:periodic}

We revisit the periodic perturbation  example we discussed in the introduction. Given $H$, the periodically perturbed Hamiltonian  $H_R$ is defined by \eqref{takisHR} for  some large $R\in \N$ with $\zeta_R$ as in  \eqref{takis8} and $\zeta$ satisfying \eqref{takiszeta}. 
\smallskip

Let $\overline H$ and $\overline H_R$ be the effective constants associated with $H$ and $H_R$ respectively. In view of Corollary \ref{cor:1}, we know that   $R^{d}(\overline H_R-\overline H)$ is bounded. Here we show that, under suitable assumptions on the unperturbed problem, this quantity has a limit. 
\smallskip

The asymptotic result is stated next. Notice that claim depends nontrivially on the dimension.

\begin{thm}\label{them:main} Assume  \eqref{takis5}, \eqref{takis6} and \eqref{takisadditional}.  If $d\geq 2$,
$$\lim_{R\to+\infty} R^d(\overline H_R-\overline H)=0.
$$
If $d=1$ and, in addition, the $R\Z^d$-periodic Hamiltonian $H_R$ also satisfies \eqref{takisadditional}, then 
$$
\lim_{R\to+\infty} R(\overline H_R-\overline H)=   -(\int_{Q} \frac{1}{D_pH(\chi'(x), x)} dx)^{-1}\int_{\text{sppt}(\zeta)} \left(H^{-1}(\zeta(x) + \overline H, x) - H^{-1}(\overline H, x) \right)dx.$$
\end{thm}

The result when $d\geq 2$ is surprising. Indeed, as discussed in the introduction, the variational representations of $\chi_R$ and $\chi$ suggest that we should have  $\overline H\approx \overline H_R-R^{-d}\int \zeta d\tilde \sigma$ with  $\int \zeta d\tilde \sigma$ positive because $\tilde \sigma$ has a full support.  The claim in the Theorem \ref{them:main} contradicts this intuition, since it implies that the optimal trajectories of the perturbed problem avoid the obstacles. \\
Of course, when $d=1$ the optimal trajectories have no room to escape and need to go through the bumps.
\smallskip

We present first the proof of the result for $d=1$, which  is rather straightforward and is based  on the  exact formulae which are available in view of the assumptions. Then we move  to the higher dimensional setting, which is more complicated and requires considerable more tools and work. 
\smallskip

\subsection*{The problem in one dimension} We present here the:
\begin{proof}[Proof of Theorem \ref{them:main} when $d=1$] For $R\in \N,$ we consider the cell problems \eqref{cp} and \eqref{cellpbHR}. In view of the assumptions, the problems have $C^{1,1}$, with bounds independent of $R$,  solutions $\chi$ and $\chi_{R}$,  which are respectively $\Z$- and $R\Z$- periodic.
\smallskip

Following the discussion in the subsection about the one-dimensional  problem,  we can also rewrite \eqref{cellpbHR}
as the ode 
\be\label{takisode1}
\chi '_{R}=H^{-1}(\zeta_R(x) +\overline H_R, x) \ \text{ in } \ \R,
\ee 
together with the condition 
\be\label{takis43}
\int_{Q_R} H^{-1}(\zeta_R(x) +\overline H_R, x) dx=0,
\ee
where $r \mapsto H^{-1}(r, x)$ is the same branch of the inverse of $H$ we used for \eqref{takisode}.
\smallskip

Let $S_R:=\overline H_R - \overline H$ and recall that, in view of  Corollary~\ref{cor:1}, $ 0 \leq R(-S_R)\leq C$.  
\smallskip

We combine  \eqref{takis41} and \eqref{takis43}  as 
\begin{equation}\label{takis40}
\int_{Q_R} \left( H^{-1}(\zeta_R(x) +\overline H_R, x) - H^{-1}(\overline H, x) \right)dx=0, 
\end{equation}
we rewrite it as 
\be\label{takis45}
\text{(I)}_R +\text{(II)}_R=0,
\ee
with 
$$\text{(I)}_R:= \int_{Q_R} \left(H^{-1}(\zeta_R(x) + S_R + \overline H, x) -  H^{-1}(\zeta_R(x) + \overline H, x)\right)dx,
 $$
and
$$\text{(II)}_R:=\int_{Q_R} \left(H^{-1}(\zeta_R(x) + \overline H, x) - H^{-1}(\overline H, x) \right)dx,$$
and we study each term separately.
\smallskip

The strict convexity of $H$ and the bound on $S_R$ yield 
$$ |H^{-1}(\zeta_R(x) + S_R + \overline H, x) -  H^{-1}(\zeta_R(x) + \overline H, x) - D_r H^{-1}(\zeta_R(x)  +\overline H, x)S_R| \lesssim (S_R)^2\lesssim R^{-2}$$  

Then
$$|\text{(I)}_R - \frac{1}{R}\int_{Q_R}  D_r H^{-1}(\zeta_R(x)  +\overline H, x)[RS_R] dx|\lesssim |Q_R||S_R|^2\lesssim   R^{-1}.
$$
Since $R$ is large and $\zeta$ has compact support, in view of the definition of $\zeta_R$, there is only one bump in $Q_R$. Hence, using  \eqref{takis42ter}, we get 
$$
\begin{array}{l}
\ds \int_{Q_R}  D_r H^{-1}(\zeta_R(x)  +\overline H, x) dx = 
 \int_{Q_R}  \left(D_r H^{-1}(\zeta (x)  +\overline H, x)  - D_r H^{-1}(\overline H, x) \right)dx  +
 \\[4.5mm] 
 \ds \int_{Q_R}  D_r H^{-1}(\overline H, x) dx= R \int_{Q}  D_r H^{-1}(\overline H, x) dx +  (\text{III})_R=
 R (\int_{Q} \frac{1}{D_pH(\chi'(x), x)} dx) + (\text{III})_R,
\end{array}
$$
with  
\[ (\text{III})_R:= \int_{Q_R} \left(D_r H^{-1}(\zeta(x)  +\overline H, x)  - D_r H^{-1}(\overline H, x) \right)dx. \] 
It is now immediate that 
\[ \left|(\text{III})_R\right| =\left|\int_{\text{sppt}(\zeta)}  \left(D_r H^{-1}(\zeta(x)  +\overline H, x)  - D_r H^{-1}(\overline H, x) \right)dx \right|\lesssim 1.\]
Collecting all the previous information we find 
\be\label{takis44}
\left|\text{(I)}_R -(\int_{Q} \frac{1}{D_pH(\chi'(x), x)} dx) RS_R\right|\lesssim R^{-1}.
\ee 
Turning out attention to $\text{(II)}_R$ we observe that, since $\zeta$ has compact support, 
\[\text{(II)}_R=\int_{\text{sppt}(\zeta)} \left(H^{-1}(\zeta(x) + \overline H, x) - H^{-1}(\overline H, x) \right)dx.\]
The claim now follows. 
\qed

\smallskip
\subsection*{The multi-dimensional problem}
The main tool  of the proof when $d\geq 2$ is  the perturbed corrector $\chi_\infty$, that is a solution to \eqref{takis2}, which, as it turns out, 
keeps tracks of the difference between $\overline H_R$ and $\overline H$; see  Lemma \ref{lem:keykey} below. The core of the argument consists in showing that the difference $\chi_\infty-\chi$ tends to a constant at infinity. This statement relies heavily on the assumption that the projected invariant measure has a full support (see Lemma \ref{lem:key} and its proof). 
\begin{lem}\label{lem:keykey} For any $t>0$, the map $x \mapsto   \left[ (\chi_\infty-\chi)(\bar \gamma^x(s)\right]_0^t$ belongs to $L^1_{\tilde \sigma}(\R^d)$ and
$$
\liminf_{R\to +\infty} R^d(\overline H_R-\overline H) \geq t^{-1} \int_{\R^d} \left[ (\chi_\infty-\chi)(\bar \gamma^x(s))\right]_0^t d \tilde \sigma(x).
$$
\end{lem}

\begin{proof} 
The variational representation formulae for viscosity solutions to convex Hamilton-Jacobi equations give, for any $t>0$,  
$$
\chi_R(x)=
\inf_{{\mathcal A}_x} \left[ \int_0^t (L(\dot \gamma(s), \gamma(s)) +\overline H_R+\zeta_R(\gamma(s))) \ ds + \chi_R(\gamma(t))\right],
$$
and 
\be\label{takisoptimalinfty}
\chi_\infty(x)=
\inf_{{\mathcal A}_x}\left[ \int_0^t (L(\dot \gamma(s), \gamma(s)) +\overline H+\zeta(\gamma(s))) \ ds + \chi_\infty(\gamma(t))\right],
\ee
where ${\mathcal A}_x$ is the set of Lipschitz curves $\gamma: [0,\infty) \to \R^d$ such that $\gamma(0)=x$,
and 
\[\chi(x)= \int_0^t (L(\dot{\bar \gamma}^x(s), \bar \gamma^x(s)) +\overline H) \ ds + \chi(\bar \gamma^x(t));\]
recall that  $L$ is defined by \eqref{def:L}, while the last equality is \eqref{takisoptimal}.
\smallskip

Hence, as in the introduction, 
$$
\begin{array}{rl}
\ds t(\overline H_R-\overline H) \;= & \ds  -\inf_{{\mathcal A}_x} \left[ \int_0^t (L(\dot \gamma(s), \gamma(s))+\zeta_R(\gamma(s))) \ ds + \chi_R(\gamma(t))\right] +\chi_R(x)\\[2.5mm]
& \qquad \ds + \int_0^t L(\dot{\bar \gamma}^x(s), \bar \gamma^x(s)) \ ds + \chi(\bar \gamma^x(t))- \chi(x,
\end{array}
$$
and, after integrating  over $Q_R$ with respect to the measure $\tilde \sigma$,
$$
\begin{array}{rl}
\ds tR^d (\overline H_R-\overline H) \;= & \ds  \int_{Q_R} \left\{ -\inf_{{\mathcal A}_x}\left[ \int_0^t (L(\dot \gamma(s), \gamma(s))+\zeta_R(\gamma(s))) \ ds + \chi_R(\gamma(t))\right] +\chi_R(x) \right. \\
& \qquad \ds \left. + \int_0^t L(\dot{\bar \gamma}^x(s), \bar \gamma^x(s)) \ ds - \chi(\bar \gamma^x(t))+ \chi(x)  \; \right\}\ d\tilde \sigma(x).
\end{array}
$$
Since the map $x\mapsto \bar \gamma^x(t)$ leaves the measure $\tilde \sigma$ invariant (on the torus) and $\chi$ and $\chi_R$ are respectively  $\Z^d$- and $R\Z^d$- periodic, for $R\in \N$, we have 
$$
\int_{Q_R} \chi(x) \ d\tilde \sigma(x) = \int_{Q_R}\chi(\bar \gamma^x(t))\ d\tilde \sigma(x) \ \ \text{and} \  \ \int_{Q_R} \chi_R(x) \ d\tilde \sigma(x) = \int_{Q_R}\chi_R(\bar \gamma^x(t))\ d\tilde \sigma(x). 
$$
Therefore 
$$
\begin{array}{rl}
\ds tR^d (\overline H_R-\overline H) \;= & \ds -\int_{Q_R} \int_0^t \zeta_R(\bar \gamma^x(s)) \ ds d\tilde \sigma(x)  \\
& \qquad \ds + \int_{Q_R} \left\{ -\inf_{{\mathcal A}_x}\left[ \int_0^t (L(\dot \gamma(s), \gamma(s))+\zeta_R(\gamma(s))) \ ds + \chi_R(\gamma(t))\right] \right. \\
& \qquad \qquad \ds \left. + \left[\int_0^t (L(\dot{\bar \gamma}^x(s), \bar \gamma^x(s))+\zeta_R(\bar \gamma^x(s)) ) \ ds + \chi_R(\bar \gamma^x(t))\right]  \ \right\} d\tilde \sigma(x).
\end{array}
$$
We now discuss the behavior, as $R\to+\infty$, of the two integrals in the righthand side of the equality above.  
\smallskip

Recall that   $\dot {\bar \gamma}^x$ is  uniformly bounded.   Therefore,  $\bar \gamma^x(s)$ does not see the bumps $\zeta(\cdot-kR)$ for $k\in \Z^d\backslash\{0\}$ as soon as $x\in Q_R$, $s\in [0,t]$ and $R$ large enough with respect to $t$. Thus, for sufficiently large $R$,  
$$
\int_{Q_R} \int_0^t \zeta_R(\bar \gamma^x(s)) \ ds d\tilde \sigma(x) = 
\int_{\R^d} \int_0^t \zeta(\bar \gamma^x(s)) \ ds d\tilde \sigma(x).
$$
For the second integral, we note that the integrand is nonnegative.  Using  Fatou's Lemma  and the convergence of $\chi_R$ to $\chi_\infty$, we find  
$$
\begin{array}{l}
\ds 0\geq t\liminf_{R\to+\infty} R^d (\overline H_R-\overline H) \;\geq \; \ds- \int_{\R^d} \int_0^t \zeta(\bar \gamma^x(s)) \ ds d\tilde \sigma(x)  \\
\qquad \qquad  \ds  +\int_{\R^d} \left\{- \inf_{{\mathcal A}_x}\left[ \int_0^t (L(\dot \gamma(s), \gamma(s))+\zeta(\gamma(s))) \ ds + \chi_\infty(\gamma(t))\right] \right. \\
\qquad \qquad  \quad \ds \left. +\left[ \int_0^t (L(\dot{\bar \gamma}^x(s), \bar \gamma^x(s))+\zeta(\bar \gamma^x(s)) ) \ ds + \chi_\infty(\bar \gamma^x(t)) \right] \right\} d\tilde \sigma(x),
\end{array}
$$
which, in view of the nonnegativity of the integrand in the second integral, yields the $\tilde \sigma$ integrability of   the map  
$$
\begin{array}{rl}
\ds x \mapsto \xi(x):= & \ds  - \inf_{{\mathcal A}_x}\left[ \int_0^t (L(\dot \gamma(s), \gamma(s))+\zeta(\gamma(s))) \ ds + \chi_\infty(\gamma(t))\right]
\\
& \ds \qquad 
+\left[ \int_0^t (L(\dot{\bar \gamma}^x(s), \bar \gamma^x(s))+\zeta(\bar \gamma^x(s)) ) \ ds + \chi_\infty(\bar \gamma^x(t)) \right]. 
\end{array}
$$

It follows from  \eqref{takisoptimalinfty} and \eqref{takisoptimal} that 
$$
\begin{array}{rl}
\ds \xi(x)\; = & \ds -\left[\chi_\infty(x)-t\overline H\right]+ \left[ \chi(x)-\chi(\bar \gamma^x(t))-t\overline H+\int_0^t \zeta(\bar \gamma^x(s))  \ ds 
+ \chi_\infty(\bar \gamma^x(t))\right]\\
= & \ds \big[ (\chi_\infty-\chi)(\bar \gamma^x(s))\big]_0^t+\int_0^t \zeta(\bar \gamma^x(s)) \ ds.
\end{array}
$$
Since the map $x\mapsto \int_0^t \zeta(\bar \gamma^x(s))  \ ds$ has  compact support, it is  
integrable with respect to  $\tilde \sigma$.

\smallskip

This last observation together with the integrability of $\xi$ yield the first assertion. The second one is immediate  from the formulae above. 
%
\end{proof}


The next lemma is about the fact that, at least along the optimal trajectories, $\chi_\infty -\chi$ has limits.  

\begin{lem}\label{lem:nondec} For $\tilde \sigma-$a.e. $x\in \R^d$, the map $t\mapsto (\chi_\infty-\chi)(\bar \gamma^x(t))$ is non decreasing on any time-interval $[t_1,t_2]$ such that $\bar\gamma^x$ does not encounter ${\rm sppt}(\zeta)$. In particular, the limits 
$$
c^\pm(x):= \lim_{t\to \pm\infty} (\chi_\infty-\chi)(\bar \gamma^x(t))
$$
exist for $\tilde \sigma-$a.e. $x\in \R^d$ and,  if $\bar \gamma^x$ never encounters ${\rm sppt}(\zeta)$, then $c^-(x)\leq c^+(x)$.
\end{lem}

\begin{proof} Let $x\in \R^d$, and $t_1, t_2 \in \R$ such that $t_1 < t_2$ and $\gamma^x([t_1,t_2]) \cap \text{sppt}(\zeta) =\emptyset$, and, hence,   $\zeta(\bar \gamma^x)=0$  on $[t_1,t_2]$. Using \eqref{takisoptimalinfty} and \eqref{takisoptimal},  for any $t_1\leq s_1\leq s_2\leq t_2$, we find
$$
\begin{array}{rl}
\ds \chi_\infty(\bar \gamma^x(s_1)) \;  \leq & \ds \int_{s_1}^{s_2} (L(\dot{\bar \gamma}^x(s), \bar \gamma^x(s)) +\overline H+\zeta(\bar \gamma^x(s))) \ ds + \chi_\infty(\bar \gamma^x(s_2)) \\[3.5mm]
= & \ds \chi(\bar \gamma^x(s_1))- \chi(\bar \gamma^x(s_2)) + \chi_\infty(\bar \gamma^x(s_2)).
\end{array}
$$
Since, for $\tilde \sigma-$a.e. $x\in \R^d$,
$$
\lim_{t\to \pm\infty} \frac{\bar \gamma^x(t)}{t}= e\neq 0, 
$$
there exists $T$ such that $\bar \gamma^x$ does not intersect ${\rm sppt}(\zeta)$ for $|t| \geq T$. Then $t\mapsto (\chi_\infty-\chi)(\bar \gamma^x(t))$ is non decreasing and bounded on the intervals $(-\infty,-T]$ and $[T, +\infty)$, and the claimed  limits exist. 
\smallskip

 If $\bar \gamma^x$ does not encounter ${\rm sppt}(\zeta)$ at all, then $t\to (\chi_\infty-\chi)(\bar \gamma^x(t))$ is non decreasing on $\R$, and, hence, 
$$
c^-(x)\leq \lim_{t\to -\infty} (\chi_\infty-\chi)(\bar \gamma^x(t))\leq \lim_{t\to +\infty} (\chi_\infty-\chi)(\bar \gamma^x(t))= c^+(x).
$$
\end{proof}
The next step is crucial, since it asserts that it is possible to always compare  $c^-$ and $c^+$. 
 \begin{lem}\label{lem:key}  The maps $c^+$ and $c^-$ are constant with $c^-\leq c^+$. 
\end{lem}
\begin{proof} 
Fix $x_0\in \R^d$ such that \eqref{RotNumber} holds. Since $\tilde \sigma$ is ergodic, there exist $t_n\to+\infty$ and $k_n\in \Z^d$ such that, as $n\to \infty$,  
$$
\left|\bar \gamma^{x_0}(t_n)-x_0-k_n\right|\to 0. 
$$
In view of the $\Z^d$-periodicity of the flow $x\mapsto \bar \gamma^x$, that is the fact that, for every $k\in \Z^d$, $\bar \gamma^{x+k}=\bar \gamma^x + k$, 
we also have, for all $k\in \Z^d$, 
\be\label{choixx0BIS}
\lim_{t\to \pm\infty} \frac{\bar \gamma^{x_0+k}(t)}{t}= e \neq 0, 
\ee
and
\be\label{choixtnBis}
\lim_{n\to \infty}\left|\bar \gamma^{x_0+k}(t_n)-x_0-k-k_n\right|=0.
\ee
The boundedness and Lipschitz continuity  of $\chi_\infty$ allow to choose a sequence $(\chi_\infty(\cdot+k_n))_{n\in \N}$ that converges  locally uniformly to some map $\chi_\infty^+$. 
\smallskip

Fix now some $k\in \Z^d$. Then \eqref{choixx0BIS} yields some  $t_0$ such that $\bar \gamma^{x_0+k}([t_0, \infty))\cap \text{sppt}(\zeta)=\emptyset,$ and, in turn, 
Lemma \ref{lem:nondec}  states that the map $t\mapsto (\chi_\infty-\chi)(\bar \gamma^{x_0+k}(t))$ is nondecreasing and converges to $c^+(x_0+k)$ as $t\to+\infty$. 
\smallskip

In particular, for any $t\in \R$,  
$$
\lim_{n\to \infty}  (\chi_\infty-\chi)(\bar \gamma^{x_0+k}(t_n))= \lim_{n\to \infty}  (\chi_\infty-\chi)(\bar \gamma^{x_0+k}(t_n+t)) = c^+(x_0+k).
$$
Then \eqref{choixtnBis} and the continuity and periodicity of $\chi$ give
$$
\lim_{n} \chi(\bar \gamma^{x_0+k}(t_n))=\lim_n \chi(x_0+k+k_n)= \chi(x_0+k), 
$$
while the uniform continuity of $\chi_\infty$ implies 
$$
\lim_{n\to \infty}  \chi_\infty(\bar \gamma^{x_0+k}(t_n)) = \lim_{n\to \infty}  \chi_\infty(x_0+k+k_n) = \chi_\infty^+(x_0+k).
$$
Note also that, using  \eqref{choixtnBis}  with  $k=0$, and the periodicity and continuity of flow $x\mapsto \bar \gamma^x$, we find that, as $n\to \infty$,
$$
\left|\bar \gamma^{x_0+k}(t_n+t)-\bar \gamma^{x_0+k}(t)-k_n\right|= \left|\bar \gamma^{\bar \gamma^{x_0}(t_n)-k_n}(t)-\bar \gamma^{x_0}(t)\right|  \to 0.
$$
In conclusion 
$$
\begin{array}{rl}
\ds c^+(x_0+k)\; = & \ds \lim_{n\to \infty} (\chi_\infty-\chi)(\bar \gamma^{x_0+k}(t_n+t))\\
= & \ds \lim_{n\to \infty} (\chi_\infty-\chi)(\bar \gamma^{x_0+k}(t)+k_n)=(\chi_\infty^+-\chi)(\bar \gamma^{x_0+k}(t)). 
\end{array}
$$
This proves that, for all $t\in \R$ and $k\in \Z^d$,  
\be\label{chiinfty+-chiCst}
 (\chi_\infty^+-\chi)(x_0+k)= (\chi_\infty^+-\chi)(\bar \gamma^{x_0+k}(t))= c^+(x_0+k). 
\ee
Next we show that the map $t\mapsto (\chi_\infty^+-\chi)(\bar \gamma^x(t))$  is constant for any $x \in \text{sppt}(\tilde \sigma).$ 
\smallskip

Fix $x\in {\rm sppt}(\tilde \sigma)$. 
Since  $\tilde \sigma$ is ergodic with full support, there exist sequences $s_n\to +\infty$ and $m_n\in \Z^d$ such that, as $n\to \infty$,  
$$
x-\bar \gamma^{x_0}(s_n)-m_n=x-\bar \gamma^{x_0+m_n}(s_n) \to 0.
$$
Then   \eqref{chiinfty+-chiCst} implies that  the map 
$$
s\mapsto (\chi_\infty^+-\chi)(\bar \gamma^{x_0+m_n}(s_n+s))=  (\chi_\infty^+-\chi)(\bar \gamma^{\bar \gamma^{x_0+m_n}(s_n)}(s))
$$ has constant on $\R$  value $c^+(x_0+m_n)$ and  converges  locally uniformly to 
 $s\to (\chi_\infty^+-\chi)(\bar \gamma^{x}(s))$, which is therefore also constant on $\R$, that is, for all $s \in \R$ and all  $ x\in {\rm sppt}(\tilde\sigma)$
 \be\label{chiinfty+-chiCstBis}
 \lim_{n\to \infty} c^+(x_0+m_n)= \; (\chi_\infty^+-\chi)(\bar \gamma^{x}(s))=  (\chi_\infty^+-\chi)(x).
 \ee
Finally,   \eqref{RotNumber} and the fact that $|k_n|\to +\infty$ imply that  $\chi_\infty^+$ solves the same equation as $\chi$, that is 
$$
H(D\chi_\infty^+,x)= \overline H\qquad {\rm in}\; \R^d,
$$
the main difference being that $\chi_\infty^+$ is not periodic a priori. 
\smallskip

The next step is to show that $D\chi_\infty^+=D\chi$ a.e. and for this we use that  $\tilde \sigma>0$ a.e.. 
\smallskip

Let $x$ be a point of differentiability of $\chi_\infty^+$ and  note $\chi$ is differentiable at $x$. Then, since $\dot{\bar \gamma}^{x}(0)= -D_pH(x, D\chi(x))$,  \eqref{chiinfty+-chiCstBis} gives 
$$
\lg D(\chi_\infty^+-\chi)(x), -D_pH(x, D\chi(x))\rg =0.
$$
On the other hand the uniform convexity of $H$ implies, for some $C>0$, 
$$
\begin{array}{rl}
\ds 0\; = & \ds  H(D\chi_\infty^+,x)-H(D\chi,x) \\[2mm]  
\geq & \ds \lg D_pH( D\chi,x), D(\chi_\infty^+-\chi)(x)\rg + C|D(\chi_\infty^+-\chi)(x)|^2 \\[2mm]
\geq & \ds C|D(\chi_\infty^+-\chi)(x)|^2.
\end{array}
$$
This proves that $D\chi_\infty^+=D\chi$ a.e. and, in particular, that $\chi_\infty^+-\chi$ is constant, which means that  $\chi_\infty^+$ is periodic. 
\smallskip

It follows from  \eqref{chiinfty+-chiCst} that  $c^+(x_0+k)=c^+(x_0)$ for any $k\in \Z^d$, which, using    \eqref{chiinfty+-chiCstBis}, leads to $c^+(x_0)= (\chi_\infty^+-\chi)(x)$ for any $x\in \R^d$. 
\smallskip
 
A symmetric construction shows that $c^-$ is also constant. 
\smallskip

Finally,  \eqref{RotNumber} yields  some  $k_0\in \Z^d$ large enough such  that the trajectory $\bar \gamma^{x_0+k_0}$ does not encounter ${\rm sppt}(\zeta)$ and, in view of  Lemma \ref{lem:nondec},  $c^-=c^-(x_0+k_0)\leq c^+(x_0+k_0)=c^+$.  Note that we use here the fact that we work in dimension $d\geq 2$, since otherwise it is not true that the trajectory $\bar \gamma^{x_0+k_0}$ does not intersect ${\rm sppt}(\zeta)$ for $k_0$ large enough. 
\end{proof}
We continue with the:

\begin{proof}[Proof of Theorem \ref{them:main} for $d\geq 2$.] Fix $r\geq 1$ large enough so that $Q_r$ contains the support of $\zeta$ and set $F:=\{\bar \gamma^x(t):   x\in Q_r, \; t\in \R\}$, that is $F$ contains all points that can be reached by optimal trajectories starting in $Q_r$ at some time $t\in \R$. 
\smallskip

The continuity of the flow $(t,x)\mapsto \bar \gamma^x(t)$  and \eqref{RotNumber} imply the existence  a time $T_0\in\N$ such that, for any $x\in Q_r$ and all $t$ such that $ |t|\geq T_0$, 
\be\label{defT0}
 \bar \gamma^x(t) \notin Q_r. 
\ee

Since $\overline H_R\leq \overline H$, Lemma \ref{lem:keykey} suggests that to conclude,  we just need to check that 
$$
I:= \int_{\R^d} \left[ (\chi_\infty-\chi)(\bar \gamma^x(s)\right]_0^{T_0} d \tilde \sigma(x) \geq 0.
$$
Lemma \ref{lem:nondec} states  that $\left[ (\chi_\infty-\chi)(\bar \gamma^x(s)\right]_0^{T_0}$ is nonnegative, if the trajectory $\bar \gamma^x$ does not encounter $\text{sppt}(\zeta)$. Since $\text{sppt}(\zeta) \subset Q_r$, $F$ in particular contains all the initial positions $x$ such that $\bar \gamma^x$ intersects the support of $\zeta$. Thus
$$
I\geq \int_{F} \left[ (\chi_\infty-\chi)(\bar \gamma^x(s)\right]_0^{T_0} d \tilde \sigma(x) .
$$
The set $F$ is not precise enough and we do not have much control over its size. In the next step, we replace it by a smaller one $\tilde F$, which carries more information,  without increasing by ``too much'' the size of the lower bound on $I$ in the inequality above.  
\smallskip

Let us recall that, from Lemma \ref{lem:nondec}, $(\chi_\infty-\chi)(\bar \gamma^x(s))$ converges to $c^\pm$ as $\pm s \to +\infty$ for a.e. $x\in Q_r$. 
So,  by Egoroff's theorem, for any 
fixed $\ep>0$,  there exists  a compact subset $K$ of $Q_r \subset F$  and a time $T\in \N$ such that 
\begin{itemize}
\item[(i)] if    $x\in K$ and $\pm s \geq T$, then  $\ds \left| (\chi_\infty-\chi)(\bar \gamma^x(s))-c^\pm \right|\leq \ep$, 
\item[(ii)] if  $\tilde F:=\{\bar\gamma^x(s): x\in K, \; s\in \R\}$, then 
$$
  \int_{F} \left[ (\chi_\infty-\chi)(\bar \gamma^x(s))\right]_0^{T_0} d \tilde \sigma(x) \geq 
   \int_{\tilde F} \left[ (\chi_\infty-\chi)(\bar \gamma^x(s))\right]_0^{T_0} d \tilde \sigma(x)-\ep
$$
and
$$
   \tilde \sigma (\tilde F\cap Q_r)\geq \tilde \sigma (F\cap Q_r)-\ep=  \tilde \sigma (Q_r)-\ep.
$$
\end{itemize}
The next step is to provide a more precise characterization for $\tilde F$. For this we construct $E\subset K$ such that 
\be\label{takis50bis}
\tilde F=\{\bar \gamma^x(t):x\in E, \; t\in \R\} \ \text{and} \  \{\bar \gamma^x(t): t\in \R \}\cap E=\{x\} \text{for all  $x\in E$}.
\ee 
For any $x\in K$,  let $\tau_x:=\inf\{t\in \R: \bar \gamma^x(t)\in K\}$ and set $E:= \{\bar \gamma^x(\tau_x): x\in Q_r\}$. It is immediate that $E$ is a Borel measurable set and satisfies \eqref{takis50bis}.
\smallskip

For $k\in \Z$, set 
$$
E(k):= \{\bar \gamma^x(t):t\in [k,k+1)T_0, \; x\in E\}.
$$
The family $(E(k))_{k\in \Z}$ is a partition of $\tilde F$. Moreover note that the definition of $E$ and \eqref{defT0} yield
$$
 \tilde F\cap Q_r \subset E(0).
$$
For $n\in \N$, $n\geq T$ and large enough, we have
$$
\begin{array}{l}
\ds   \int_{\tilde F} \left[ (\chi_\infty-\chi)(\bar \gamma^x(s)\right]_0^{T_0}\ d \tilde \sigma(x) \\
\qquad  \ds \geq 
 \sum_{k=-n-1}^{n} \int_{E(k)} \left[ (\chi_\infty-\chi)(\bar \gamma^x(s)\right]_0^{T_0}\ d \tilde \sigma(x) -\ep \\
\qquad  \ds = 
 \sum_{k=-n-1}^{n} \left[\int_{E(k)}(\chi_\infty-\chi)(\bar \gamma^x(T_0))\ d \tilde \sigma(x) 
 -\int_{E(k)}(\chi_\infty-\chi)(x)\ d \tilde \sigma(x)\right] -\ep \\ 
 \ds \qquad =
  \sum_{k=-n-1}^{n} \left[\int_{E(k+1)}  (\chi_\infty-\chi)(x)\ d \tilde \sigma(x)
- \int_{E(k)}  (\chi_\infty-\chi)(x) \ d \tilde \sigma(x)\right]-\ep, 
  \end{array}
 $$
 where we used that $\tilde \sigma$ is invariant by the flow $\bar \gamma^x$ and  the image by  the map   $x\mapsto \bar \gamma^x(T_0)$ of $E(k)$ is $E(k+1)$. 
\smallskip

Hence
$$
\begin{array}{l}
\ds   \int_{\tilde F} \left[ (\chi_\infty-\chi)(\bar \gamma^x(s)\right]_0^{T_0}\ d \tilde \sigma(x)\\
 \ds \qquad \geq \; 
   \int_{E(T+1)}  (\chi_\infty-\chi)(x)\ d \tilde \sigma(x)
- \int_{E(-T-1)}  (\chi_\infty-\chi)(x) \ d \tilde \sigma(x) -\ep.
  \end{array}
 $$
The choice  of $\tilde F$ (property (i)) gives 
$$
\begin{array}{rl}
\ds   \int_{E(T+1)}  (\chi_\infty-\chi)(x)\ d \tilde \sigma(x)
\; = &
\ds   \int_{E(0)}  (\chi_\infty-\chi)(\bar \gamma^x(T+1))\ d \tilde \sigma(x) \\
 \geq & \ds (c^+-\ep) \tilde \sigma(E(0)),
  \end{array}
$$
and, similarly,  
$$
 \int_{E(-T-1)}  (\chi_\infty-\chi)(x)\ d \tilde \sigma(x) \leq (c^-+\ep)  \tilde \sigma(E(0)). 
$$
Thus, since  $c^+\geq c^-$ and $ \tilde F\cap Q_r \subset E(0)$, we get 
$$
\begin{array}{l}
\ds   \int_{\R^d} \left[ (\chi_\infty-\chi)(\bar \gamma^x(s)\right]_0^{T_0}\ d \tilde \sigma(x)\\
 \ds \qquad \geq \; (c^+-c^-)  \tilde \sigma(E(0)) - C\ep\;  \geq \; 
 (c^+-c^-)  \tilde \sigma(\tilde F\cap Q_r) -C\ep .
  \end{array}
 $$
Using property (ii) in the definition of $\tilde F$, we finally find 
$$
\begin{array}{l}
\ds I=  \int_{\R^d} \left[ (\chi_\infty-\chi)(\bar \gamma^x(s)\right]_0^{T_0}\ d \tilde \sigma(x)
\geq (c^+-c^-)  \tilde \sigma(Q_r)- C\ep.
  \end{array}
 $$
Letting $\ep\to0$, we obtain 
$$
0\geq I\geq (c^+-c^-)  \tilde \sigma(Q_r)  \geq 0, 
$$
which yields both   $I=0$ and $c^+=c^-$.  
\end{proof}

\subsection*{The proof of Corollary~\ref{cor:cor}}  
Let $c$ be the common value of $c^\pm$;  recall that in the  proof of Theorem \ref{them:main} we showed that $c^+=c^-$.  Let now $x\in {\mathcal O}$, the open set of points such that $(\bar \gamma^x(t))_{t\in \R}$ does not intersects ${\rm sppt}(\zeta)$.  Lemma \ref{lem:nondec} implies that  the map $t\mapsto (\chi_\infty-\chi)(\bar \gamma^x(t))$ is non increasing and has the same limit $c$ at $-\infty$ and $+\infty$. Hence it is constant in $\mathcal O$ and (i) holds. 
\smallskip

Next we show that, for any $\ep>0$, there exists $T_\ep>0$ such that 
$$
(\chi_\infty-\chi)(\bar \gamma^x(t))\geq c-\ep \ \text{ for all $x \in {\rm sppt}(\zeta)$ and $ t\geq T_\ep.$}
$$
Fix $x\in  {\rm sppt}(\zeta)$. We know from Lemma \ref{lem:nondec} that the map $t\mapsto (\chi_\infty-\chi)(\bar \gamma^x(t))$ is non decreasing for $t$ large enough and converges to $c$. Hence,  there exists $T_x$ such that  $\left|(\chi_\infty-\chi)(\bar \gamma^x(t))-c\right|\leq \ep/2$ for $t\geq T_x$. 
\smallskip

Sine the map $y\to \bar \gamma^y(t)$ is continuous, we also have $\left|(\chi_\infty-\chi)(\bar \gamma^y(T_x))- c\right|\leq \ep$ for any $y$ in some ball centered at $x$ and of radius $\delta_x>0$. 
\smallskip

Thus, since the map $t\mapsto (\chi_\infty-\chi)(\bar \gamma^y(t))$ is nondecreasing,  for $y\in B(x,\delta_x)$ and $t\geq T_x$ we have  $(\chi_\infty-\chi)(\bar \gamma^y(t))\geq  c-\ep$. 
We conclude using  a standard compactness argument. 
\smallskip

Next we claim that,  
for any $\ep>0$, there exists $K_\ep\geq 0$ such that, for all  $x\in \R^d$ such that $\lg x, e\rg \leq -K_\ep$,
\be\label{chiinfty-chi+}
 (\chi_\infty-\chi)(x)\geq  c- \ep. 
\ee
Fix $\ep>0$, let $T_\ep$ be defined as in  the previous step and choose 
$$
K_\ep\geq MT_\ep+ \sup_{y\in {\rm sppt}(\zeta)}\lg y,e\rg,
$$ 
with  $M:=\|D_pH(D\chi,\cdot)\|_\infty$; notice that  $\|\dot {\bar \gamma}^x\|_\infty\leq M$. 
\smallskip

It follows from Corollary~\ref{cor:takis} that, if $K_\ep$ is large enough, then for any  $x\in \R^d$ with $\lg x, e\rg \geq K_\ep$, the trajectory $\bar \gamma^x(t)$ does not instersect  ${\rm sppt}(\zeta)$ for $t\geq 0$.
\smallskip

Fix $x\in \R^d$ with $\lg x, e\rg \geq K_\ep$. If $\bar \gamma^x(\R) \cap {\rm sppt}(\zeta)=\emptyset$, we know that $(\chi_\infty-\chi)(x)= c$. 
\smallskip

We now assume that there exists $t\in \R$ such that $\bar \gamma^x(t)\in {\rm sppt}(\zeta)$. Then, by the definition of $T_\ep$, $t\leq 0$, and, since as $\lg x, e\rg\geq K_\ep$ and $ \bar \gamma^x(t)\in {\rm sppt}(\zeta)$, we must have 
$$
|x-\bar \gamma^x(t)| \geq \lg x-\bar \gamma^x(t), e \rg \geq 
K_\ep- \sup_{y\in {\rm sppt}(\zeta)}\lg y,e\rg \ \geq \ \|\dot{\bar \gamma}^x\|_\infty T_\ep,
$$
and, hence,  $|t|\geq  T_\ep$. But then the first claim of the corollary gives
$$
(\chi_\infty-\chi)(x)=(\chi_\infty-\chi)(\bar \gamma^{\bar \gamma^x(t)}(-t))\geq c-\ep,
$$
which proves \eqref{chiinfty-chi+}. 
\smallskip

A similar argument  shows that, for all $\ep>0$, there exists $K_\ep\geq 0$ such that, for all $x\in \R^d$ such that $ \lg x, e\rg \leq -K_\ep$, 
\be\label{chiinfty-chi-}
(\chi_\infty-\chi)(x)\leq c+\ep, 
\ee   
which  implies (iv). 
\smallskip

We now prove (ii).  Fix $x\in \R^d$ and let $\tilde \gamma$ be the optimal path for $\chi_\infty(x)$ for positive times. Then
\be\label{eq:ppdTER}
\chi_\infty(x)= \int_0^t \left( L(\dot{\tilde \gamma}(s), {\tilde\gamma}(s))   +\zeta(\tilde \gamma(s) + \overline H \right) ds +\chi_\infty(\tilde \gamma(t)), 
\ee
while 
\be\label{eq:ppdTERbis}
\chi(x)\leq  \int_0^t \left( L(\dot{\tilde \gamma}(s), {\tilde\gamma}(s)) +\overline H\right) ds +\chi(\tilde \gamma(t)).
\ee
Fix $\ep>0$ and let $K_\ep$ be given by the previous step. It follows from  \eqref{eq:ppdTER} that, for all $t\geq 0$,  
$$
\int_0^t \left( L(\dot{\tilde \gamma}(s), {\tilde\gamma}(s))+\overline H\right) ds \leq 2\|\chi_\infty\|_\infty,
$$
which, in view of Lemma \ref{lem:avoid}, implies that there exists $T_\ep>0$ such that $\lg \tilde \gamma(t), e\rg \geq K_\ep$ for all  $t\geq T_\ep$. Then the definition of $K_\ep$ gives that,  for all  $t\geq T_\ep,$
%
$$
(\chi_\infty-\chi)(\tilde \gamma(t))\geq c-\ep. 
$$   
Subtracting \eqref{eq:ppdTERbis} from \eqref{eq:ppdTER} and using  $\zeta\geq 0$ yields 
$$
(\chi_\infty-\chi)(x) \geq \int_0^{T_\ep}\zeta(\tilde \gamma(s)) ds +(\chi_\infty-\chi)(\tilde \gamma(T_\ep))
\geq c-\ep,
$$
which proves (ii) since  $\ep$ is arbitrary. 
 \smallskip

 It remains to show (iii).  Let $K$ be given by Lemma \ref{lem:avoid} for $C:= 2\|\chi_\infty\|_\infty$ and choose $x\in \R^d$ such that  $\lg x,e\rg \geq K$. 
 \smallskip
 
It follows, again  from Lemma \ref{lem:avoid},  that the trajectory $\bar \gamma^x$ avoids the support of $\zeta$ for all positive times. Then Lemma \ref{lem:nondec} implies that the map $t\mapsto (\chi_\infty-\chi)(\bar \gamma^x(t))$ is non decreasing on $[0,+\infty)$, so that 
$$
(\chi_\infty-\chi)(x) \leq \lim_{t\to+\infty} (\chi_\infty-\chi)(\bar \gamma^x(t))= c. 
$$
Since, in view of  (ii), the opposite inequality always holds, (iii) is proved. 
\end{proof}

\section{Sharper convergence for the random perturbation}\label{sec:random}
Given a  $\Z^d$-periodic Hamiltonian satisfying \eqref{takis5}, \eqref{takis6} and \eqref{takisadditional}, 
we consider the random perturbation $H_\eta$  defined  in  \eqref{takis60}, with  $\zeta:\R^d\to [0,+\infty)$  and  $(X^\eta_k)_{k\in \Z^d}$ as in \eqref{takiszeta} and \eqref{takis100} respectively.
\smallskip

Let  $\overline H_\eta$ be the effective constant associated with $H_\eta$.  We are interested  in the behavior  of the ratio $(\overline H_\eta-\overline H)/\eta$ as $\eta\to 0$; recall that  in Corollary  \ref{cor:1bis} we proved that this ratio is bounded. 
\smallskip
%

We prove that $\lim_{\eta \to 0} (\overline H_\eta-\overline H)/\eta$ exists and equals a non-zero constant 
when $d=1$ and $0$ when $d\geq 2$.

\begin{thm}\label{thm:main2} Assume  \eqref{takis5}, \eqref{takis6}, \eqref{takisadditional}, \eqref{takis60}, \eqref{takiszeta} and  \eqref{takis100}.  When $d=1$, 
$$
\lim_{\eta\to0^+}\frac{\overline H-\overline H_\eta}{\eta}= -(\int_{Q} \frac{1}{D_pH(D\chi, x)} dx)^{-1}\int_{\text{sppt}(\zeta)} \left(H^{-1}(\zeta(x) + \overline H, x) - H^{-1}(\overline H, x) \right)dx,$$
and, when $d \geq 2$, 
$$
\lim_{\eta\to0^+}\frac{\overline H-\overline H_\eta}{\eta}=0.
$$
\end{thm}
%
A discussion similar to the one after Theorem~\ref{them:main} explains the difference between the one and the multi-dimensional case. 
\smallskip

The proof for $d=1$ makes strong use of the fact that in this case there exist  ``almost explicit formulae'' for 
$\overline H$ and $\overline H_\eta$ and follows along the lines of the proof of the $d=1$ limit in  Theorem~\ref{them:main}.
\smallskip

\subsection*{\it Proof of Theorem~\ref{thm:main2}  for $d=1$}  For $\eta>0$ we consider the cell problems  \eqref{cp}  and 
\be\label{takis300}
H_\eta(\chi_{\eta,x}, x)=\overline H_\eta \ \text{in} \  \R,
\ee
with $H_\eta$ as in \eqref{takis60}. We remark that, since we work on the real line and given the assumptions on $H$, \eqref{takis300} has a strictly sublinear at infinity solution; see Lions and Souganidis  \cite{lions2003correctors}. Moreover,  the solutions  of  \eqref{cp} and  \eqref{takis300}  are in $C^{1,1}$ with bounds independent of $\eta$. Finally we recall that the solution $\chi$ of \eqref{cp} is  $\Z-$ periodic.
\smallskip

Following the discussion in the subsection about the one-dimensional  problem as well as the proof of Theorem
~\ref{them:main}  for $d=1$,   we  rewrite \eqref{takis300}
as the ode 
\be\label{takis301}
\chi' _{\eta}(x)=H^{-1}(\zeta_\eta(x) +\overline H_\eta, x) \ \text{ in } \ \R,
\ee 
together with condition 
\be\label{takis302}
\E\int_{Q} H^{-1}(\zeta_\eta(x) +\overline H_\eta, x) dx=0,
\ee
where $r \to H^{-1}(r, x)$ is the same branch for the inverse of $H$ we used for \eqref{takisode}.
\smallskip

Let $S_\eta:=\overline H_\eta - \overline H$ and recall that, in view of  Corollary~\ref{cor:1}, $ 0 \leq -S_\eta\leq C\eta$.  
\smallskip

Combining  \eqref{takis41} with $R=1$ and \eqref{takis302},   we find
\begin{equation}\label{takis303}
\E\int_{Q}  H^{-1}(\zeta_\eta (x) +\overline H_\eta, x)=\int_Q H^{-1}(\overline H, x) dx.
\end{equation}
\smallskip

In what follows, for simplicity we assume that $\text{sppt}(\zeta) \subset Q$. Otherwise we need to account for lower order terms in $\eta^2$; we leave the details to the reader.
\smallskip

Since, in view of \eqref{takis61} and the assumed Bernoulli law, the left hand side of \eqref{takis302} can be evaluated explicitly, we rewrite  \eqref{takis303} as
\be\label{takis304}
(1-\eta) \int_Q  H^{-1}(S_\eta + \overline H, x)dx + \eta\int_Q  H^{-1}(\zeta(x) +S_\eta + \overline H, x)dx= \int_Q H^{-1}(\overline H, x)dx.
\ee

The strict convexity of $H$ and the bound on $S_\eta$ in Corollary~\ref{cor:1bis} yield, as in the proof of Theorem~\ref{them:main} for $d=1$,  
$$ |H^{-1}(\zeta (x) +S_\eta + \overline H, x) -  H^{-1}(\zeta(x) + \overline H, x) - D_r H^{-1}(\zeta (x)  +\overline H, x)S_\eta| \lesssim (S_\eta)^2\lesssim \eta^{2}$$  
and 
$$ |H^{-1}(S_\eta + \overline H, x) -  H^{-1}( \overline H, x) - D_r H^{-1}(\overline H, x)S_\eta | \lesssim (S_\eta)^2\lesssim \eta^{2}$$  
Integrating over $Q$ and using that 
$$\int_Q\left(D_rH^{-1}(\zeta(x) +S_\eta +\overline H,x) - D_rH^{-1}(\overline H,x)\right)dx \lesssim 1$$
in \eqref{takis304},  we get 
$$
\eta \left[ \int_Q(H^{-1}(\zeta (x) +\overline H,x)-H^{-1}(\overline H,x))dx\right] + S_\eta \left[\int_Q D_rH^{-1}(\overline H,x))dx + \text{O}(\eta)\right] = \text{O}(S_\eta^2),$$
and, since $ |S_\eta| =\text{O}(\eta),$ as $\eta \to 0$,
$$ \frac{S_\eta}{\eta} \to  -(\int_{Q} \frac{1}{D_pH(D\chi, x)} dx)^{-1}\int_{\text{sppt}(\zeta)} \left(H^{-1}(\zeta(x) + \overline H, x) - H^{-1}(\overline H, x) \right)dx.$$
\qed
\subsection*{The multidimensional random problem}
In higher dimensions, that is when $d\geq 2$,  the proof is rather delicate and more involved. It is based on constructing  a suitable (random) trajectory $\gamma$ along which it is possible to control the quantity $\int_0^t (L(\dot \gamma, \gamma)+\zeta_\eta(\gamma)+\overline H)ds$.   This is accomplished combining the optimal trajectories  of the unperturbed and the ``one bump'' problems.

The proof is divided  into four parts.  In the first we introduce some notation, in the second we explain the construction of the random trajectory and    
in the third we provide the key estimates. 
The argument  is completed in the fourth part. 
%
%
\subsection*{Notation}\label{subset.notation}
In what follows we denote by $\bar \gamma^x$ and $\tilde \gamma^x$ the Borel measurable with respect to $x$ optimal paths  for $\chi$ and $\chi_\infty$, that is, for any $x\in \R^d$ and any $0\leq s \leq t$, 
$$
\chi(\bar \gamma^x (s))= \int_{s}^t \left( L(\dot {\bar \gamma}^x(\tau),\bar \gamma^{x}(\tau))+\overline H\right)d\tau +\chi(\bar \gamma^x(t))m
$$
and 
$$
\chi_\infty (\tilde \gamma^x (s))= \int_{s}^t \left( L(\dot {\tilde\gamma}^x(\tau),\tilde \gamma^{x}(\tau))+\overline H+\zeta(\tilde \gamma^x(\tau)\right) d\tau +\chi_\infty(\tilde \gamma^x(t)).
$$
Set
$$
 \zeta_\infty(z):= \sum_{k\in \Z^d} \zeta(z-k),
 $$
and note for later use that 
$$
\E\left[ \zeta_\eta(x)\right] = \eta \zeta_\infty (x). 
$$
Fix  $D>0$ be such that ${\rm sppt}(\zeta)\subset B(0,D)$ and  $\ep>0$. Then   Lemma \ref{lem:avoid} and Corollary \ref{cor:cor} imply the existence of  $R_0\geq D$ and $T_0>0$ such that 
\be \label{toto1}
 \chi_\infty (x)= \chi(x)+c \text { for any $x\in \R^d$ with $\lg x,e\rg \geq R_0$},
\ee
\be\label{toto2}
(\chi_\infty-\chi)(x) \leq c+\ep \ \text{ for any $x\in \R^d$ with $\lg x,e\rg \leq -R_0$,} 
\ee
and,  if $\gamma$ is a trajectory such that 
$$
\int_0^t \left( L(\dot \gamma, \gamma)+ \overline H\right)ds\leq 2 (\|\chi\|_\infty+\|\chi_\infty\|_\infty),
$$
then, for all $  t\geq 0$,
\be\label{bla1}
(i)~\inf_{s\geq t} \lg \gamma(s)-\gamma(t), e\rg \geq -R_0 \quad \text{and}  \quad 
(ii)~\inf_{s\geq t+T_0} \lg \gamma(s)-\gamma(t), e\rg \geq R_0. 
\ee


We also set 
\be\label{defR0prime}
R_0':= R_0+ \|D_pH(D\chi,\cdot)\|_\infty T_0,
\ee


and, for $s<t$, $E_s^-$, $E_t^+$ and $E_{s,t}$ are the sets 
$$\begin{cases}
E_s^-:=\left\{y\in \R^d: \lg y,e\rg <s\right\}, 
\;
E_t^+:=\left\{y\in \R^d: \lg y,e\rg \geq t\right\},\; \text{and} \\[1mm] 
E_{s,t}:=\left\{y\in \R^d:s\leq \lg y,e\rg <t\right\}.
\end{cases}$$
Finally,  ${\mathcal F}_s^-$, 
${\mathcal F}_t^+$ and ${\mathcal F}_{s,t}$ are the $\sigma-$algebras generated by the random variables $X_k$ with $k\in \Z^d$ and $k\in E_s^-$, $k\in E_t^+$ and $k\in E_{s,t}$ respectively. 

\smallskip
In what follows we fix  $R\geq 2R_0'$, $n\in \N$ and $\eta>0$ and  we set, for all $n\in \N,$
$
r_n:= nR. 
$
\smallskip

Throughout the proof $C$ is a generic constant, which may change from line to line, depends on the data  and on $\ep$ by the choice of  $R_0$, $T_0$ or $R_0'$, but not on $R$, $\eta$ or $n$. In addition $C_R$ is  a constant which may also depend on $R$.

\subsection*{The random trajectory}\label{subset.deftraj}
We  construct  by induction a sequence of random  points  $(x_n)_{n\in \N}$ and times $(\tau_n)_{n\in \N}$ and, then, on each time interval $[\tau_n,\tau_{n+1}]$, a random trajectory $\gamma$. 
\smallskip

We define $\tau_0$, $x_0$ and $\gamma$ on $[0,\tau_0]$ as follows:
$$
\tau_0:=\inf\left\{t\geq 0, \; \inf_{s\geq t} \lg \bar \gamma^0(s),e\rg \geq r_0\right\}, \ x_0:= \bar \gamma^0(\tau_0) \ \text{and} \ \gamma:= \bar \gamma^0 \ \text{on  \ $[0,\tau_0]$.}
$$
Note that $x_0$, $\tau_0$ are deterministic with $\gamma(0)=0$, $\lg \gamma(\tau_0),e\rg=r_0=0$. 

\smallskip

Assuming next  that $x_n$ and $\tau_n$ are known, we find $x_{n+1}$, $\tau_{n+1}$ and $\gamma$ on $[\tau_n, \tau_{n+1}]$. For this we need to 
consider three  disjoint  events $A_{n,0}$, $A_{n,1}$ and $A_{n,2}$. Roughly speaking, in the event $A_{n,0}$, the perturbation $\zeta_\eta$ vanishes on the trajectory $\bar \gamma^{x_n}$ in the set  $E_{r_n + R_0', r_{n+1}-R_0}$. In $A_{n,1}$, the trajectory $\bar \gamma^{x_n}$ encounters only one bump  in   $E_{r_n + R_0', r_{n+1}-R_0}$, and there are no other bump in a large neighborhood of the trajectory. The last event is the complement of the other two. 
\smallskip

More precisely, recalling that $r_n=nR$ and that the radius $D$ is such that ${\rm sppt}(\zeta)\subset B(0,D)$, 
the event
$A_{n,0}$ is defined by the property that there is no $k\in \Z^d\cap E_{r_n + R_0', r_{n+1}-R_0}$ with $X_k=1$ and $|k-\bar \gamma^{x_n}(t)|\leq D$ for some $t\geq0$. In the event $A_{n,1}$, there exists a unique $\hat k_n\in \Z^d\cap E_{r_n + R_0', r_{n+1}-R_0}$ with $X_{\hat k_n}=1$ and $|\hat k_n-\bar \gamma^{x_n}(t)|\leq D$ for some  $t\geq 0$, but there is no other $k'\in \Z^d\cap E_{r_n + R_0', r_{n+1}-R_0} \cap B(x_n,K_R)$ such that $X_{k'}=1$; here $K_R>0$ is a large constant depending on $R$ defined
in \eqref{defKR} below.
Finally,   
$
A_{n,2}= \Omega\backslash ( A_{n,0}\cup A_{n,1}).
$
\smallskip

In $A_{n,0}\cup A_{n,2}$, 
$$
\tau_{n+1}:= \tau_n+\inf \{t\geq 0: \inf_{s\geq t} \lg \bar \gamma^{x_n}(s),e\rg\geq r_{n+1}\}, \quad  x_{n+1}:=\bar \gamma^{x_n}(\tau_{n+1}-\tau_n)
$$
and
$$
 \gamma:= \bar \gamma^{x_n}(\cdot-\tau_n) \ \text{in} \  [\tau_n,\tau_{n+1}]. 
$$
\smallskip

In $A_{n,1}$, we set $\gamma:= \bar \gamma^{x_n}(\cdot -\tau_n)$ in $[\tau_n, \tau_n+T_0]$, $z_n:=  \bar \gamma^{x_n}(T_0)$,  
$$
\tau_{n+1}:= \tau_n+T_0+ \inf \{t\geq 0: \inf_{s\geq t} \lg \tilde \gamma^{z_n-\hat k_n}(s)+\hat k_n, e\rg \geq r_{n+1}\},
$$
where $\hat k_n$ is given in the definition of $A_{n,1}$, 
$$
x_{n+1}:= \tilde \gamma^{z_n-\hat k_n}(\tau_{n+1}-\tau_n-T_0)+ \hat k_n \  \text{and} \  \gamma:= \tilde \gamma^{z_n-\hat k_n}(\cdot -\tau_n-T_0)+ \hat k_n \ \text{in} \  [\tau_n+T_0,\tau_{n+1}]. 
$$
Note that, by definition, for all $n\in \N$, 
$$
\gamma(0)=0, \; \gamma(\tau_n)=x_n\; {\rm and}\;  \lg x_n, e\rg=r_n.
$$
Moreover, by Corollary \ref{cor:takis}, the times $\tau_n$ are finite. 

\subsection*{The key properties of the construction}\label{subsec:keyppties} In the next four lemmata we study the properties of the construction above that are needed to complete the proof of the theorem. 

\begin{lem}\label{lem.bargammaErn} For any $n\in \N$ and all $t>0$, the trajectory $\bar \gamma^{x_n}$ remains in $E_{r_n}^+$. 
\end{lem}
\begin{proof} We argue using induction. Observe that, in view of the choice  of $x_0$ and $\tau_0$, the claim is immediate for $n=0$, 
and assume that the result holds for some $n$. 
\smallskip

The choice of $\tau_{n+1}$ 
in $A_{n,0}\cup A_{n,2}$, yields  that, for all $t\geq 0$, 
 $\bar \gamma^{x_{n+1}}(t)= \bar \gamma^{x_n}(t+\tau_{n+1})$ belongs to $E_{r_{n+1}}^+$. 
\smallskip

In view of the definition of $\tau_{n+1}$ in $A_{n,1}$, to conclude we need to show that, for all $t\geq \tau_{n+1}-\tau_n-T_0,$ 
\be\label{claimtilde=}
 \tilde \gamma^{z_n-\hat k_n}(t)+\hat k_n= \bar \gamma^{x_{n+1}}(t-\tau_{n+1}+\tau_n+T_0). 
\ee
Note first that  \eqref{claimtilde=} holds for $t= \tau_{n+1}-\tau_n-T_0$. Then 
the definition of $\tau_{n+1}$ and the fact that $\hat k_n\in E_{r_n+R_0',r_{n+1}-R_0}$ yield that, for all $t\geq \tau_{n+1}-\tau_n-T_0, $
\be\label{tildegammaegeqR0}
\lg \tilde \gamma^{z_n-\hat k_n}(t), e\rg \geq r_{n+1}-\lg \hat k_n , e\rg \geq R_0.
\ee
It follows,  in view of \eqref{toto1}, that for all $t\geq \tau_{n+1}-\tau_n-T_0, $
\be\label{chiinfty=chi}
\chi_\infty( \tilde \gamma^{z_n-\hat k_n}(t)) = \chi(\tilde \gamma^{z_n-\hat k_n}(t)).
\ee
The optimality of $\tilde \gamma^{z_n-\hat k_n}$ for $\chi_\infty$  implies  that, for any $t\geq \tau_{n+1}-\tau_n-T_0$, 
$$
\chi_\infty(x_{n+1})= \int_{\tau_{n+1}-\tau_n-T_0}^t ( L(\dot {\tilde\gamma}^{z_n-\hat k_n}, \tilde \gamma^{z_n-\hat k_n})+\overline H+\zeta(\tilde \gamma^{z_n-\hat k_n})) ds +\chi_\infty(\tilde \gamma^{z_n-\hat k_n}(t)), 
$$
where,   by \eqref{tildegammaegeqR0},  $\zeta(\tilde \gamma^{z_n-\hat k_n}(s)) =0$ on  $[\tau_{n+1}-\tau_n-T_0,t]$.
\smallskip

Hence, using \eqref{chiinfty=chi}, we find 
$$
\chi(x_{n+1})= \int_{\tau_{n+1}-\tau_n-T_0}^t \left( L(\dot {\tilde\gamma}(s),\tilde \gamma^{z_n-\hat k_n}(s))+\overline H\right) ds +\chi(\tilde \gamma(t)).
$$
It follows that  $\tilde \gamma^{z_n-\hat k_n}$ is optimal for $\chi$ in $[\tau_{n+1}-\tau_n-T_0, +\infty)$, and, in view of the  uniqueness of the optimal solution, we obtain that, for all  $ t\geq \tau_{n+1}-\tau_n-T_0,$
$$
\tilde \gamma^{z_n-\hat k_n}(t) = \bar \gamma^{\tilde \gamma^{z_n-\hat k_n}(\tau_{n+1}-\tau_n-T_0)}(t-\tau_{n+1}+\tau_n+T_0). 
$$
Using that  the map $x\to \bar \gamma^x$ is periodic, we get that,  for any $t\geq \tau_{n+1}-\tau_n-T_0$, 
$$
\tilde \gamma^{z_n-\hat k_n}(t) +\hat k_n= \bar \gamma^{\tilde \gamma^{z_n-\hat k_n}(\tau_{n+1}-\tau_n-T_0)+\hat k_n}(t-\tau_{n+1}+\tau_n+T_0)
=\bar \gamma^{x_{n+1}}(t-\tau_{n+1}+\tau_n+T_0),
$$
which is \eqref{claimtilde=}.
\end{proof}

\begin{lem}\label{lem.measurable} For any $n\in \N_0$, $x_n$, $\tau_n$ and the restriction of $\gamma$ to $[0, \tau_n]$ are ${\mathcal F}^-_{r_n-R_0}$-measurable, while the events $A_{n,0}$, $A_{n,1}$ and $A_{n,2}$ are ${\mathcal F}^-_{r_{n+1}-R_0}$-measurable.
\end{lem}
\begin{proof} We argue again by induction. The claim is true for $n=0$
since $x_0$, $\tau_0$ and the restriction of $\gamma$ to $[0, \tau_0]$ are deterministic. 
\smallskip

We assume next that $x_n$, $\tau_n$ and the restriction of $\gamma$ to $[0, \tau_n]$ are ${\mathcal F}^-_{r_n-R_0}$-measurable. 
Knowing $x_n$ and $\tau_n$, it follows that 
$A_{n,0}$, $A_{n,1}$ and $A_{n,2}$ belong to ${\mathcal F}^-_{r_{n+1}-R_0}$, and the induction assumption,  
implies that $A_{n,0}$, $A_{n,1}$ and $A_{n,2}$ are ${\mathcal F}^-_{r_{n+1}-R_0}$-measurable. It follows from  their definition that $x_{n+1}$, $\tau_{n+1}$ and the restriction of $\gamma$ to $[0, \tau_{n+1}]$ are ${\mathcal F}^-_{r_{n+1}-R_0}$-measurable. 
\end{proof}

\begin{lem} For any $n\in \N_0$, $\gamma(t)$ belongs to $E^+_{r_n}$ for all $t\geq \tau_n$.
\end{lem}

\begin{proof} It is enough to show that, for any $n\in \N_0$, $\gamma(t)$ belongs to $E^+_n$ for $t\in[ \tau_n,\tau_{n+1}]$.
Fix $n\in \N$. Lemma \ref{lem.bargammaErn} implies that, for all $t\geq 0$,  $\bar \gamma^{x_n}(t)\in E_{r_n}^+$, and the claim is clear in $A_{n,0}\cup A_{n,2}$. In $A_{n,1}$, we have $\gamma(t)= \bar \gamma^{x_n}(t-\tau_n)$ for $t\in [\tau_n,\tau_n+T_0]$, and again $\gamma(t)\in  E_{r_n}^+$ in this interval. Moreover, \eqref{bla1}(ii) yields that 
$$
\lg z_n,e \rg = \lg \bar \gamma^{x_n}(T_0),e \rg \geq \lg x_n,e\rg + R_0= r_n+R_0,
$$
and, hence, the choice of $R_0$ in \eqref{bla1} yields  
$$
\inf_{s\geq 0} \lg \tilde \gamma^{z_n-\hat k_n}(s)+\hat k_n, e\rg \geq \lg z_n,e\rg -R_0=  r_n. 
$$
\end{proof}
\begin{lem}\label{lem:limtaun} There exists $C_0>0$, which is  independent of $\ep$, $R$, $n$ and $\eta$,  such that,  for all $n\in \N$, $C_0^{-1}R\leq \tau_{n+1}-\tau_n\leq C_0R$. In particular, as $n\to \infty,$ and almost surely,  $\tau_n\to +\infty.$ 
\end{lem}

\begin{proof} Since $\lg x_n,e\rg = \lg \gamma(\tau_n), e\rg =  n R$, we deduce that 
$$
\begin{array}{rl}
\ds R = r_{n+1}-r_n \; = & \ds  \lg \gamma(\tau_{n+1})- \gamma (\tau_n), e\rg 
\leq  \ds (\tau_{n+1}- \tau_n) \|\dot \gamma\|_\infty\\[1mm]
 \leq & \ds (\tau_{n+1}- \tau_n) (\|D_pH(\cdot, D \chi)\|_\infty+\|D_pH(\cdot, D \chi_\infty)\|_\infty), 
\end{array}
$$
which proves that $(\tau_{n+1}- \tau_n)\geq C_0^{-1}R$ for some $C_0>0$. 
\smallskip

Recall that, for all $x\in \R^d$ and $t\geq 0$, 
$$
\int_0^t \left( L({\dot {\bar \gamma}}^x,\bar \gamma^x)+ \overline H\right)ds\leq 2 \|\chi\|_\infty. 
$$
Then Lemma \ref{lem:avoid} yields $\bar T_0>0$ such that, for all  $t\geq0$, 
$$
\inf_{s\geq t+\bar T_0} \lg \bar \gamma(s)-\bar \gamma(t), e\rg \geq 1.
$$
Similarly, since, for all $t\geq 0$,
$$
\int_0^t \left( L(\dot {\tilde \gamma}^x, \tilde \gamma^x)+ \overline H\right)ds\leq 2\|\chi_\infty\|_\infty,
$$
it follows that
$$
\inf_{t\geq 0} \inf_{s\geq t+\bar T_0} \lg \tilde \gamma(s)-\tilde \gamma(t), e\rg \geq 1.
$$
Hence, in the event $A_{n,0}\cup A_{n,2}$, where $\gamma(t)= \bar \gamma^{x_n}(t-\tau_n)$ for $t \in [\tau_n, \tau_{n+1}]$, we have 
$$
R= \lg \gamma(\tau_{n+1})-\gamma(\tau_n), e\rg \geq [(\tau_{n+1}-\tau_n)/\bar T_0],
$$
where $[a]$ denotes the integer part of $a$. 

\smallskip
In the event $A_{n,1}$,  $\gamma= \bar \gamma^x(\cdot-\tau_n)$ on $[\tau_n, \tau_n+T_0]$ and 
$\gamma= \tilde \gamma^{z_n}(\cdot-\tau_n-T_0)$ on $[\tau_n+T_0, \tau_{n+1}]$, where $z_n= \bar \gamma(\tau_n+T_0)$. Thus
$$
\lg \gamma(\tau_{n+1})-\gamma(\tau_n+T_0), e\rg \geq [(\tau_{n+1}-\tau_n-T_0)/\bar T_0],
$$
while 
$$
\lg \gamma(\tau_{n}+T_0)-\gamma(\tau_n), e\rg \geq [T_0/\bar T_0].
$$
Since, from the first part of the proof,  $\tau_{n+1}-\tau_n$ is large for large $R$, combining the  inequalities above, we  find that, for   a suitable choice of $C_0$, 
$$
R \geq C_0^{-1}(\tau_{n+1}-\tau_n).
$$
\end{proof}

We  now define the constant $K_R$ that was used in the construction of the random trajectory as 
\be\label{defKR}
K_R:= C_0R \left( \|D_pH(\cdot, D\chi)\|_\infty+ \|D_pH(\cdot, D\chi_\infty)\|_\infty\right),
\ee
and remark, for later use, that, in $[\tau_n,\tau_{n+1}]$, $|\gamma(t)-x_n|\leq K_R$. 
\smallskip

We also emphasize that the construction of $C_0$ in the proof of Lemma \ref{lem:limtaun} is  deterministic. indeed, it does not depend on the definition of the random sets $A_{n,0}$, $A_{n,1}$ and $A_{n,2}$, but only on the possible expressions the trajectory $\gamma$ can take in these events. In particular, the definition of $K_R$ is not circular. 
\begin{lem}\label{lem.PAn1} There exists $C_1>0$, which is independent of $\ep$, $R$, $n$ and $\eta$, such that 
$$
\P\left[ A_{n,1}\right]\leq C_1R\eta.
$$
\end{lem}

The intuition behind the lemma is quite clear. The set  $A_{n,1}$ is contained in the event that there is at least one bump $\zeta(\cdot-k)$, which is both  near the trajectory $\gamma$ and belongs to $E_{ r_n,r_{n+1}}$. Since $|k|\lesssim R$,   $\P[A_{1,n}]\lesssim R\eta.$ 

\begin{proof}[The proof of  Lemma~\ref{lem.PAn1}] Let $S$ be the random set
$$
S:= \left\{ k\in \Z^d\cap E_{r_n,r_{n+1}}: \inf_{t\geq 0} |\bar \gamma^{x_n}(t)-k|\leq D\right\}.
$$
Lemma \ref{lem:avoid} implies the existence of a constant $\bar T_0$, independent of $\ep$, such that 
$$
\inf_{s\geq t+ \bar T_0} \lg \bar \gamma^x(s)-\bar \gamma^x(t),e\rg \geq 1.
$$

Next we discretize $\bar \gamma^x$ with step size $t_l = l \bar T_0$, for some  $l\in \N$, such that 
$$
|\bar \gamma^{x_n}(t)-\bar \gamma^{x_n}(l\bar T_0)|\leq M\bar T_0,
$$
where $M:= \|D_pH(D\chi, \cdot)\|_\infty$.
\smallskip

It follows that 
$$
S\subset \left\{ k\in \Z^d\cap E_{r_n,r_{n+1}}: \inf_{l\in \N} |\bar \gamma^{x_n}(l\bar T_0)-k|\leq D+M\bar T_0\right\}.
$$
Note also that since, for any $l\geq L_R:=[R+D+M\bar T_0]+2$, we have 
$$
\lg \bar \gamma^{x_n}(l \bar T_0), e\rg > \lg x_n, e\rg +R+D+ M\bar T_0 +1\geq r_{n+1}+D + M\bar T_0+1,
$$
if $k\in \Z^d$ is such that   $ |\bar \gamma^{x_n}(l\bar T_0)-k|\leq D+M\bar T_0$,  then $k\in E^+_{r_n+1}$. 
\smallskip

Hence
$$
S\subset \left\{ k\in \Z^d\cap E_{r_n,r_{n+1}}: \inf_{l= 0, \dots, L_R} |\bar \gamma^{x_n}(t_l)-k|\leq D+ M\bar T_0 \right\}.
$$
Since
$$
A_{n,1}\subset \left\{ \exists k\in S, \; X_k=1\right\}, 
$$
we deduce that 
$$
\begin{array}{rl}
\ds \P\left[A_{n,1}\right]\; \leq & \ds \P\left[ \exists k \in \Z^d\cap E_{r_n,r_{n+1}}, \; \exists l= \{0, \dots, L_R\}, \;  |\bar \gamma^{x_n}(t_l)-k|\leq D+ M\bar T_0, \; X_k=1  \right]\\
\leq & \ds \sum_{l=0}^{L_R}  \P\left[\exists k\in  \Z^d\cap E_{r_n,r_{n+1}}, \;    |\bar \gamma^{x_n}(t_l)-k|\leq D+ M\bar T_0, \; X_k=1  \right].
\end{array}
$$
The  ${\mathcal F}^-_{r_n-R_0}$-measurability of $x_n$ implies that  the event $\{X_k=1\}$ with  $k\in \Z^d\cap E_{r_n,r_{n+1}}$ is independent of $x_n$, and, thus 
$$
\P\left[A_{n,1}\right]\leq L_R C_d (D+ M\bar T_0)^d \eta,
$$
where $C_d$ depends only on the dimension. 
\smallskip

Then, since $L_R=[R+D+M\bar T_0]+2$, we may conclude.
\end{proof}

\subsection*{Proof of Theorem \ref{thm:main2}}\label{subsec:proof}

It is known (see \cite{S}) that, if 
$$
\theta(t):= \inf_{\xi(0)=0}  \int_0^t \left(L(\dot \xi(s),\xi(s)) + \overline H+ \zeta_\eta(\xi(s))\right)ds, 
$$
then, almost surely, 
$$
\lim_{t\to +\infty} \frac{\theta(t)}{t}= \overline H-\overline H_\eta. 
$$
In particular, in view of Lemma \ref{lem:limtaun}, we have 
$$
\lim_{n\to +\infty} \E\left[\frac{\theta(\tau_n)}{\tau_n}\right]= \overline H -\overline H_\eta .
$$
If $\gamma$ is the trajectory built in a previous subsection, 
then 
\be\label{barHetabound}
\overline H-\overline H_\eta\leq \limsup_{n\to+\infty} \E\left[\frac{1}{\tau_n}\int_0^{\tau_n} \left(L(\dot\gamma,\gamma) +  \overline H+\zeta_\eta(\gamma)\right)ds\right].
\ee
To estimate the right-hand side the inequality above, which  is the core of the proof, we need to establish three more auxiliary results, which we formulate next as separate lemmata.
\smallskip

We set 
$$
I_n:= \int_{\tau_n}^{\tau_{n+1}} \left(L(\dot\gamma,\gamma) + \overline H+\zeta_\eta(\gamma)\right)ds,
$$
and we successively estimate $I_n$ in $A_{n,0}$, $A_{n,2}$ and $A_{n,1}$ noting that the estimate in the last set is the hardest to establish.
\begin{lem}\label{lem.An0}
There exists a nonnegative random variable $R_{n,0}$ such that 
$$
\ds {\bf 1}_{A_{n,0}} I_n \; = \; \ds {\bf 1}_{A_{n,0}}  (\chi(x_n)- \chi(x_{n+1}))+ R_{n,0} \ \ {\rm with } \ \  \E\left[ R_{n,0} \right]\leq C\eta.
$$
\end{lem} 

\begin{proof} Recall that $\gamma= \bar \gamma^{x_n}(\cdot-\tau_n)$ in $A_{n,0}$.
It follows that 
$$
\begin{array}{rl}
\ds {\bf 1}_{A_{n,0}} I_n\; = & \ds {\bf 1}_{A_{n,0}} \left(\int_{0}^{\tau_{n+1}-\tau_n} \left(L(\dot{\bar \gamma}^{x_n}(s),\bar \gamma^{x_n}(s)) +  \overline H+\zeta_\eta(\bar \gamma^{x_n}(s)) \right)ds\right)\\
= & \ds  {\bf 1}_{A_{n,0}} \left( \chi(x_n)- \chi(x_{n+1}) + \int_{0}^{\tau_{n+1}-\tau_n}\zeta_\eta(\bar \gamma^{x_n}(s)) ds\right) .
\end{array}
$$
Let 
$$
R_{n,0}:=  {\bf 1}_{A_{n,0}} \int_{0}^{\tau_{n+1}-\tau_n}\zeta_\eta(\bar \gamma^{x_n}(s)) ds.
$$
Since $A_{n,0}$ is the event that there is no $k\in \Z^d\cap E_{r_n+R_0', r_{n+1}-R_0}$ with $X_k=1$ and $|k-\bar \gamma^{x_n}(t)|\leq D$ for some $t\geq0$, 
$$
\zeta_\eta(\gamma(t)) = 0 \; {\rm if }\; \gamma(t)= \bar \gamma^{x_n}(t-\tau_n)\in E_{r_n+R_0'+D, r_{n+1}-R_0-D}\; {\rm and}\; t\in [\tau_n,\tau_{n+1}]. 
$$
Let 
$$
\bar \tau_{n+1}:= \tau_n+\inf \{t\geq 0,\; \inf_{s\geq t} \lg \bar \gamma^{x_n}(s),e\rg\geq r_{n+1}\},
$$
and recall that, in view of Lemma \ref{lem.measurable},  $x_n$ and $\tau_n$ are ${\mathcal F}^-_{r_n-R_0}$ measurable. 
Then $\bar \tau_{n+1}$ is also  ${\mathcal F}^-_{r_n-R_0}$-measurable  and  equals  $\tau_{n+1}$  in $A_{n,0}$. 
\smallskip

We now estimate for how long the trajectory $\bar \gamma^{x_n}$ remains in $E_{r_n+R_0'+D, r_{n+1}-R_0-D}$. Using   \eqref{bla1}(ii) we find that, if  $C:= (1+[(R_0'+D)/R_0])$, then, for all  $t\geq CT_0$,
$$
\lg \bar \gamma^{x_n}(t)-x_n, e\rg \geq  R_0'+D,
$$
and, similarly, if $t\leq \bar \tau_{n+1}-CT_0$, then 
$$
\lg x_{n+1}-\bar \gamma^{x_n}(t), e\rg \geq  R_0+D. 
$$
 It follows that  $\bar \gamma^{x_n}(t)\in E_{r_n+R_0'+D, r_{n}-R_0-D}$ for $t\in [CT_0,\bar \tau_{n+1}-CT_0]$, and,  hence,
$$
\E\left[ {\bf 1}_{A_{n,0}}  \int_{\tau_n}^{\tau_{n+1}} \zeta_\eta (\gamma(t))dt\right] \leq
\E\left[ \int_{0}^{CT_0} \zeta_\eta(\bar \gamma^{x_n}(t))dt +  \int_{\bar \tau_{n+1}-\tau_n-CT_0}^{\bar \tau_{n+1}-\tau_n} \zeta_\eta(\bar \gamma^{x_n}(t))dt\right] .
$$
Using that  $\bar \gamma^{x_n}$ remains in $E^+_{r_n}$ for positive times,  $x_n$ is  ${\mathcal F}^-_{r_n-R_0}$-measurable and the map $\zeta_\eta$ is independent of ${\mathcal F}^-_{r_n-R_0}$  in $E^+_{r_n}$ with $\E[\zeta_\eta(z)]= \eta \zeta_\infty(z)$, we find
$$
\begin{array}{rl}
\ds \E\left[ \int_{0}^{CT_0} \zeta_\eta(\bar \gamma^{x_n}(t))dt  \right] \;
= & \ds 
\E\left[ \int_{0}^{CT_0} \E\left[  \zeta_\eta(\bar \gamma^{x}(t))\ |\ {\mathcal F}^-_{r_n-R_0}\right]_{x=x_n}dt \right] \\
= & \ds 
\E\left[ \eta \int_{0}^{CT_0} \zeta_\infty(\bar \gamma^{x_n}(t))dt\right] 
\leq  \ds C\eta\|\zeta_\infty\|_\infty T_0.
\end{array}
$$
In the same way, since $x_n$, $\tau_n$ and $\bar \tau_{n+1}$ are  ${\mathcal F}^-_{r_n-R_0}$ measurable, 
$$
\begin{array}{rl}
\ds \E\left[ \int_{\bar \tau_{n+1}-\tau_n-CT_0}^{\bar \tau_{n+1}-\tau_n} \zeta_\eta(\bar \gamma^{x_n}(t))dt  \right] 
\leq  \ds C\eta\|\zeta _\infty\|_\infty T_0.
\end{array}
$$
The claim follows.
\end{proof}

\medskip

Next we estimate $I_n$ in $A_{n,2}$. 
\begin{lem}\label{lem.An2} There exists a nonnegative random variable $R_{n,2}$ such that 
$$
\ds {\bf 1}_{A_{n,2}} I_n \; = \; \ds
 {\bf 1}_{A_{n,2}}  \left(\chi(x_n)- \chi(x_{n+1})\right)+ R_{n,2} \ \ {\rm and} \ \  \E\left[R_{n,2}\right]\leq C_R\eta^2.
$$
\end{lem}

\begin{proof} It is immediate that 
$$
\begin{array}{rl}
\ds {\bf 1}_{A_{n,2}} I_n\; = \;  \ds  {\bf 1}_{A_{n,0}} \left( \chi(\gamma(x_n)- \chi(x_{n+1})\right) + R_{n,2},
\end{array}
$$
with
$$
R_{n,2}:= {\bf 1}_{A_{n,2}} \int_{0}^{\tau_{n+1}-\tau_n}\zeta_\eta(\bar \gamma^{x_n}(s)) ds .
$$
 In the event $A_{n,2}$, there are at least two different $k,k'\in \Z^d\cap E_{r_n + R_0',  r_{n+1}-R_0}\cap B(x_n, K_R)$ with $X_k=X_{k'}=1$. Since $x_n$ is independent of ${\mathcal F}_{r_n}$,  
$$
\P\left[ A_{n,2}\right] \leq C_R \eta^2.
$$
It follows from Lemma \ref{lem:limtaun} that $\tau_{n+1}-\tau_n\leq C_0R$, and, hence, 
$$
\begin{array}{rl}
\ds \E\left[ R_{n,2}\right]  \; = 
& \ds \E\left[ {\bf 1}_{A_{n,2}}  \int_{0}^{\tau_{n+1}-\tau_n} \zeta_\eta (\bar \gamma^{x_n}(t))dt\right]\\
 \leq & \ds \E\left[ {\bf 1}_{A_{n,2}}   \int_{0}^{\tau_{n+1}-\tau_n} \zeta_\infty (\bar \gamma^{x_n}(t))dt \right] \leq  C_0 R \|\zeta_\infty\|_\infty\P\left[A_{n,2}  \right],
\end{array}
$$
and the claim follows.
\end{proof}


The next lemma is about  $A_{n,1}$.
\begin{lem}\label{lem.An1} Let  $C_1$ be as in Lemma \ref{lem.PAn1}. There exists a nonnegative random variable $R_{n,1}$ such that  
$$
{\bf 1}_{A_{n,1}} I_n \leq  {\bf 1}_{A_{n,1}} \left( \chi(x_n)-\chi(x_{n+1})\right)+R_{n,1} \ \ \text{and} \ \  \E\left[R_{n,1}\right] \leq (C+C_1R\ep)\eta+C_R \eta^2.
$$
\end{lem}

\begin{proof} In the event $A_{n,1}$,  we have set  $\gamma= \bar \gamma^{x_n}(\cdot-\tau_n)$ on $[\tau_n, \tau_n+T_0]$ and $\gamma= \tilde \gamma^{z_n-\hat k_n}(\cdot-\tau_n-T_0)+ \hat k_n$ on $[\tau_n+T_0,\tau_{n+1}]$, where $z_n:=  \bar \gamma^{x_n}(T_0)$. 
So we can write 
$$
{\bf 1}_{A_{n,1}} I_n= A+B
$$
with
$$
A:= {\bf 1}_{A_{n,1}} \int_{0}^{T_0}  \left(L(\dot{\bar \gamma}^{x_n}(s),\bar \gamma^{x_n}(s)) +  \overline H+\zeta_\eta(\bar \gamma^{x_n}(s)) \right)ds,  
$$
and 
$$
B:= {\bf 1}_{A_{n,1}} \int_{0}^{\tau_{n+1}-\tau_n-T_0} \left(L(\dot{\tilde \gamma}^{z_n-\hat k_n}(s),\tilde \gamma^{z_n-\hat k_n}(s)+\hat k_n) +  \overline H+\zeta_\eta(\tilde \gamma^{z_n-\hat k_n}(s)+\hat k_n) \right)ds.  
$$
To estimate $A$ we argue  as in Lemma \ref{lem.An0}. Indeed, 
$$
A= {\bf 1}_{A_{n,1}} \left( \chi(x_n)-\chi(z_n)\right)+ R_{n,1}'
$$
with 
$$
R_{n,1}':={\bf 1}_{A_{n,1}} \int_{0}^{T_0} \zeta_\eta(\bar \gamma^{x_n}(s)) ds . 
$$
It then follows, as in the proof of Lemma \ref{lem.An0}, that
$$
\E\left[R_{n,1}'\right]\leq C\eta.   
$$
We now turn to the estimate for $B$.  In the event $A_{n,1}$, there exists a unique $\hat k_n\in \Z^d\cap E_{r_n + R_0', r_{n+1}-R_0}$ with $X_{\hat k_n}=1$ and $|\hat k_n-\bar \gamma^{x_n}(t)|\leq D$ for some  $t\geq 0$, and there is no other $k'\in \Z^d\cap E_{r_n + R_0', r_{n+1}-R_0}\cap B(x_n,K_R)$ such that $X_{k'}=1$. Therefore 
\be\label{feta=}
\zeta_\eta(x)= \zeta(x-\hat k_n) \  {\rm if } \  x\in E_{r_n + R_0'+D, r_{n+1}-R_0-D} \ \text{and} \ |x-x_n|\leq K_R.
\ee
The definition of $\tilde \gamma^{z_n-\hat k_n}$ yields
$$
\begin{array}{l}
\ds \int_{0}^{\tau_{n+1}-\tau_n-T_0} \left(L(\dot{\tilde  \gamma}^{z_n-\hat k_n}(s),\tilde \gamma^{z_n-\hat k_n}(s)) +  \overline H + \zeta(\tilde \gamma^{z_n-\hat k_n}(s)) \right)ds \\
\qquad \ds= \chi_\infty(z_n-\hat k_n)-\chi_\infty(\tilde \gamma^{z_n-\hat k_n}(\tau_{n+1}-\tau_n-T_0))
= \chi_\infty(z_n-\hat k_n)-\chi_\infty(\gamma(\tau_{n+1}-\hat k)).
\end{array}$$
Moreover, since $\hat k_n\in E_{r_n + R_0', r_{n+1}-R_0}$ while $|z_n-x_n|\leq \|D_pH(\cdot, D\chi)\|_\infty T_0$, using the  $R_0'$ in \eqref{defR0prime}, we have 
$$
\lg z_n-\hat k_n , e\rg \leq \lg x_n-\hat k_n, e\rg +  \|D_pH(\cdot, D\chi)\|_\infty T_0 \leq -R_0' +\|D_pH(\cdot, D\chi)\|_\infty T_0\leq -R_0.
$$
It then follows from \eqref{toto2} and the periodicity of $\chi$  that 
$$
 \chi_\infty(z_n-\hat k_n)\leq \chi(z_n-\hat k_n)+c+\ep = \chi(z_n)+c+\ep.
$$
The definition of $\tau_{n+1}$ implies that  $\lg \gamma(\tau_{n+1}),e\rg=r_{n+1}$ and, hence, 
$$
\lg \gamma(\tau_{n+1})-\hat k_n, e\rg \geq r_{n+1}-(r_{n+1}-R_0)= R_0,
$$
and, in view of \eqref{toto1} and the periodicity of $\chi$, 
$$
\chi_\infty(\gamma(\tau_{n+1})-\hat k_n) = \chi(\gamma(\tau_{n+1})-\hat k_n)+c= \chi(x_{n+1})+c.
$$
Collecting the above inequalities  and using the periodicity in space of $L$, we find that, in $A_{n,1}$, 
\be\label{khj;zbasjd}
\begin{array}{l}
\ds \int_{\tau_n+T_0}^{\tau_{n+1}} \left( L(\dot \gamma(s), \gamma(s))+ \zeta(\gamma(s)-\hat k_n)\right)ds  \\[1.5mm]
\ds =   \int_{0}^{\tau_{n+1}-\tau_n-T_0} \left(L(\dot{\tilde  \gamma}^{z_n-\hat k_n}(s),\tilde \gamma^{z_n-\hat k_n}(s)) +  \overline H + \zeta(\tilde \gamma^{z_n-\hat k_n}(s)) \right)ds \\[1.5mm]
\qquad \ds  \leq  \chi(z_n)-\chi(x_{n+1})+ \ep,
\end{array}
\ee
and, hence, 
$$
B\leq  {\bf 1}_{A_{n,1}} \left(\chi(z_n)-\chi(x_{n+1})+\ep\right)+ R_{n,1}'',
$$
where 
$$
R_{n,1}'':={\bf 1}_{A_{n,1}}  \int_{\tau_n+T_0}^{\tau_{n+1}} \left(\zeta_\eta(\gamma (s))- \zeta(\gamma(s)-\hat k_n)\right)ds.
$$
Recalling \eqref{feta=} as well as the fact that $\gamma(t)\in B(x_n, K_R)$ for $t\in [\tau_n,\tau_{n+1}]$, we find
$$
R_{n,1}''={\bf 1}_{A_{n,1}}  \int_{\tau_n+T_0}^{\tau_{n+1}} {\bf 1}_{\gamma \notin E_{r_n + R_0'+D, r_{n+1}-R_0-D}} \zeta_\eta(\gamma(t))dt.
$$
Let $A_{n,1}'\subset A_{n,1}$ be the event that there is at least one bump different from $\hat k_n$, in $(E_{r_n, r_n+R_0'}\cup E_{r_{n+1}-R_0, r_{n+1}+ R_0})\cap B(x_n, K_R)$. Since $\gamma |_{[\tau_n, \tau_{n+1}]}$ belongs to $E_{r_n, r_{n+1}+R_0}\cap B(x_n, K_R)$, we get 
$$
\E\left[R_{n,1}''\right]=\E\left[{\bf 1}_{A_{n,1}'}  \int_{\tau_n+T_0}^{\tau_{n+1}} {\bf 1}_{\gamma \notin E_{r_n + R_0'+D, r_{n+1}-R_0-D}} \zeta_\eta(\gamma(t))dt \right].
$$
It follows from  \eqref{bla1}(ii) that, if $t\geq  CT_0$ with $C:=([(R_0'+D)/R_0]+1)$, then  
$$
\lg \tilde \gamma^{z_n-\hat k_n}(t)+\hat k_n,e\rg \geq \lg z_n,e\rg+ ([(R_0'+D)/R_0]+1)R_0 \geq r_n +R_0'+D.
$$
Moreover, since $\lg \tilde \gamma^{z_n-\hat k_n}(\tau_{n+1}-\tau_n-T_0)+\hat k_n, e\rg =r_{n+1}$, we also find, for $t\leq \tau_{n+1}-CT_0$,  
$$
\begin{array}{rl}
\ds \lg \gamma(t),e\rg \; =  & \ds\lg \tilde \gamma^{z_n-\hat k_n}(t-\tau_n-T_0)+\hat k_n, e\rg \\
\leq & \ds \lg \tilde \gamma^{z_n-\hat k_n}(\tau_{n+1}-\tau_n-T_0)+\hat k_n, e\rg - R_0-D= r_{n+1}-R_0-D,
\end{array}
$$
and, thus $\gamma(t)\in  E_{r_n + R_0'+D, r_{n+1}-R_0-D}$ for $t\in [\tau_n+CT_0, \tau_{n+1}-CT_0]$. 
\smallskip

It follows that 
$$
\begin{array}{rl}
\ds \E\left[R_{n,1}''\right]\; \leq & \ds \E\left[{\bf 1}_{A_{n,1}'}  \int_{\tau_n+T_0}^{\tau_{n}+CT_0} \zeta_\eta(\gamma(t))dt \right]+
\E\left[{\bf 1}_{A_{n,1}'}  \int_{\tau_{n+1}-CT_0}^{\tau_{n+1}} {\bf 1}_{\gamma \notin E_{r_n + R_0', r_{n+1}-R_0}} \zeta_\eta(\gamma(t))dt \right] \\[1.3mm]
\leq & \ds \|\zeta_\eta\|_\infty CT_0 \P\left[ A_{n,1}'\right].
\end{array}
$$
In view of the fact that  in $A_{n,1}'$ there exist  at least two distinct bumps in the set $E_{r_n, r_{n+1}+R_0}\cap B(x_n, K_R)$, we have 
$\ds \E\left[R_{n,1}''\right] \leq C_R \eta^2.$

\smallskip

Writing 
$$
R_{n,1}= R_{n,1}'+R_{n,1}''+\ep{\bf 1}_{A_{n,1}}, 
$$
we obtain 
$$
{\bf 1}_{A_{n,1}} I_n \leq  {\bf 1}_{A_{n,1}} \left( \chi(x_n)-\chi(x_{n+1})\right)+ R_{n,1},
$$
and, in view of  Lemma \ref{lem.PAn1} and the above estimates, 
$$
\E\left[R_{n,1}\right]\leq C\eta+C_R \eta^2+\ep\P\left[A_{n,1}\right]\leq (C+C_1R\ep)\eta+C_R \eta^2. 
$$
\end{proof}

\smallskip

We complete now the proof. 
\begin{proof}[Proof of Theorem \ref{thm:main2}(continued)]  Combining Lemma \ref{lem.An0}, Lemma \ref{lem.An2} and Lemma \ref{lem.An1}, we find
$$
I_n\leq  \chi(x_n)-\chi(x_{n+1})+ R_n
$$
where $\ds R_n:= R_{n,0}+R_{n,1}+R_{n,2}$, and, for all  $n\in \N$,
$$
\E\left[R_n\right]\leq  (C+C_1R\ep)\eta+C_R \eta^2.
$$
Therefore 
$$
\begin{array}{l}
\ds \int_{\tau_0}^{\tau_n} \left(L(\dot\gamma(t), \gamma(t)) +  \overline H+\zeta_\eta(\gamma(t))\right)dt
\; = \; \ds \sum_{k=0}^{n-1} I_k  
\leq  \ds   \chi(x_0)-\chi(x_n)+\sum_{k=0}^{n-1} R_k. 
\end{array}
$$
It follows from  Lemma \ref{lem:limtaun} that 
$$
C_0^{-1} Rn \leq \tau_n\leq C_0Rn,
$$
and, thus, 
$$
\begin{array}{l}
\ds \limsup \E\left[\frac{1}{\tau_n}\int_{\tau_0}^{\tau_n} \left(L(\dot\gamma(t),\gamma(t)) +  \overline H+\zeta_\eta(\gamma(t))\right)ds\right]  \\[1mm]
\qquad \qquad \ds \leq \ds \limsup_n \frac{C_0}{Rn}\left( 2\|\chi\|_\infty+ n(C+C_1R\ep)\eta+C_R \eta^2n \right)\\[1.5mm]
\qquad \qquad \ds \leq \; C_0(CR^{-1}+C_1\ep)\eta+\tilde C_R \eta^2.
\end{array}
$$
Then, in view of \eqref{barHetabound}, we get 
$$
\overline H -\overline H_\eta \leq C_0(CR^{-1}+C_1\ep)\eta+C_R \eta^2, 
$$
and, thus, 
$$
\limsup_{\eta\to 0^+} \frac{\overline H -\overline H_\eta }{\eta} \leq C_0(CR^{-1}+C_1\ep).
$$
Since  $C_0$ and $C_1$ are independent on $\ep$ and  $C_0$, $C_1$ and $C$ are independent on $R$, letting first $R\to+\infty$ and then $\ep\to0$, we conclude that 
$$
\limsup_{\eta\to 0^+} \frac{\overline H -\overline H_\eta}{\eta} \leq 0.
$$
The inequality  $\overline H -\overline H_\eta \geq0$ then completes the proof. 
\end{proof}

\newcommand{\noop}[1]{}

\end{document}